\documentclass[11pt]{amsart}
\usepackage{amsmath}
\usepackage{amsthm}
\usepackage{amssymb}
\usepackage{amscd}
\usepackage{amsfonts}
\usepackage{amsbsy}
\usepackage{epsfig,afterpage}
\usepackage[dvips]{psfrag}
\usepackage{subfigure}
\usepackage{epsfig}
\usepackage{amsmath}
\usepackage{srcltx}
\usepackage{amsfonts}
\usepackage{amssymb}
\usepackage{psfrag}
\usepackage{verbatim}
\usepackage{graphicx}
\usepackage{amsmath}
\usepackage{amsthm}
\usepackage{amssymb}
\usepackage{amscd}
\usepackage{amsfonts}
\usepackage{amsbsy}
\usepackage{epsfig}
\usepackage[dvips]{psfrag}
\usepackage{multirow}
\usepackage{graphics}
\usepackage{epstopdf}
\usepackage{overpic}
\usepackage{float}

\usepackage{graphicx}

\def\e{\varepsilon}

\newcommand{\s}{\ensuremath{\mathbb{S}}}

\newtheorem {theorem} {Theorem}%[section]
\newtheorem {proposition} [theorem]{Proposition}
\newtheorem {corollary} [theorem]{Corollary}
\newtheorem {lemma}  [theorem]{Lemma}

\newtheorem {remark} [theorem]{Remark}

\newtheorem {definition} [theorem]{Definition}
\newcommand{\R}{\mathbb{R}}

\newcommand{\D}{\ensuremath{\mathbb{D}}}

\usepackage{float}
\usepackage{tikz, wrapfig}
\tikzset{node distance=3cm, auto}
\allowdisplaybreaks

\begin{document}

\title[Invariant Algebraic Surfaces and Constrained Systems]
	{Invariant Algebraic Surfaces and Constrained Systems}
	
		\author[Paulo R. da Silva and O. H. Perez ]
	{ Paulo R. da Silva and Ot\'{a}vio H. Perez  }

	\address{  S\~{a}o Paulo State University (Unesp), Institute of Biosciences, Humanities and
		Exact Sciences. Rua C. Colombo, 2265, CEP 15054--000. S. J. Rio Preto, S\~ao Paulo,
		Brazil.}
	
	\email{paulo.r.silva@unesp.br}
	
	\email{otavioperez@hotmail.com}

	\thanks{ .}


	\date{}
	\dedicatory{}
		\maketitle
	\begin{abstract}
	
We study flows of smooth vector fields $X$  over invariant surfaces $M$ which are levels of rational first integrals.
It leads us to  study constrained systems, that is, systems with impasses.
We identify a subset $\mathcal{I} \subset M$  which we call ``pseudo-impasse" set  and
analyze the flow of $X$ by points of $\mathcal{I}$. Systems well known in the literature exemplify our results: 
Lorenz, Chen, Falkner-Skan and Fisher-Kolmogorov.
We also study 1-parameter families of integrable systems and unfolding of minimal sets. Our main tool is 
the geometric singular perturbation theory.
		
	\end{abstract}

\section{Introduction}

Let $X:\R^3\rightarrow\R^3$  be a smooth vector field. A trajectory of $X$ is a smooth curve $\varphi(t)$ satisfying that 
$$\dot{\varphi}(t)=X(\varphi(t)),t\in I\subseteq \R.$$

We say that a smooth function $H:\R^3\rightarrow\R$ is a \emph{first integral} of $X$ if it satisfies $\langle \nabla H, X \rangle = 0$. 
It means that  $H(\varphi(t))=H(\varphi(0))$, for any $t\in I$.

Our main goal is to describe the flow of $X$ on invariant algebraic surfaces, that is  on $M = H^{-1}(0)$ 
with $H$ being a polynomial function. 

Here we focus our attention on surfaces $M = H^{-1}(0)$ with 
$H(x,y,z) = f(x,y)z - g(x,y),$
where $f,g$ are polynomials. Considering these surfaces is not very restrictive. In fact, as we will see in the examples below, many surfaces, including singular parts or being disconnected, can be represented in this way.

The surface $M$ can be written as the disjoint union of two subsets: $\mathcal{G}_{M}$ and 
$\mathcal{I}_{M}$. The first one is the graphic of  $z= (g/f)(x,y)$, and the second one represents the 
subset of $M$ that cannot be written as a graphic. The set $\mathcal{I}_{M}$ is called \emph{pseudo-impasse set}. 
Geometrically, $\mathcal{I}_{M}\subset M$ is a set of lines which are parallel to the $z$-axis. 
See \textit{Proposition \ref{prop-superficie}}, section 3.

Before we present our results, let's start with an example to indicate what kind of problem we are interested in.

\noindent \textbf{Example.} Consider the vector field $X(x,y,z)=(2y,2xz,(x^{2} - 1)z - (y^{2} - 1))$
where $M = H^{-1}(0)$ defines a smooth algebraic invariant surface, with $H(x,y,z) = (x^{2} - 1)z - (y^{2} - 1)$. 
We can write $M = \mathcal{G}_{M}\cup\mathcal{I}_{M}$, where $\mathcal{G}_{M} = \{(x,y,z)\in M | z  = {(y^{2} - 1)}/{(x^{2} - 1)}, x \neq\pm1\}$ and
$\mathcal{I}_{M} = \{(x,y,z)\in M | y^{2} - 1 = 0 = x^{2} - 1, z\in\mathbb{R}\}.$
$\mathcal{G}_{M}$ is a graphic  and the pseudo impasse  $\mathcal{I}_{M}$ is a set of four lines 
 that are ortogonally projected on the $xy$-plane. See figure \ref{fig-exemplo-intro}.
The flow of $X$ on $\mathcal{G}_{M}$ is described by
\begin{equation}\label{exe-intro-impasse}
\dot{x} = 2y, \ (x^{2} - 1)\dot{y} = 2x(y^{2} - 1).
\end{equation}
Additional effort is needed to describe the flow in $\mathcal{I}_{M}$. 
The projections of $\mathcal{I}_{M}$ on $xy$-plane are hyperbolic equilibrium points of the system
$\dot{x} = (x^{2} - 1)2y,$ $\dot{y} = 2x(y^{2} - 1).$
It is easy to see that $\mathcal{I}_{M}$ does not contain any 
equilibrium point and the four lines are not invariant.

\begin{figure}[h!]\label{fig-exemplo-intro}
	\center{\includegraphics[width=0.35\textwidth]{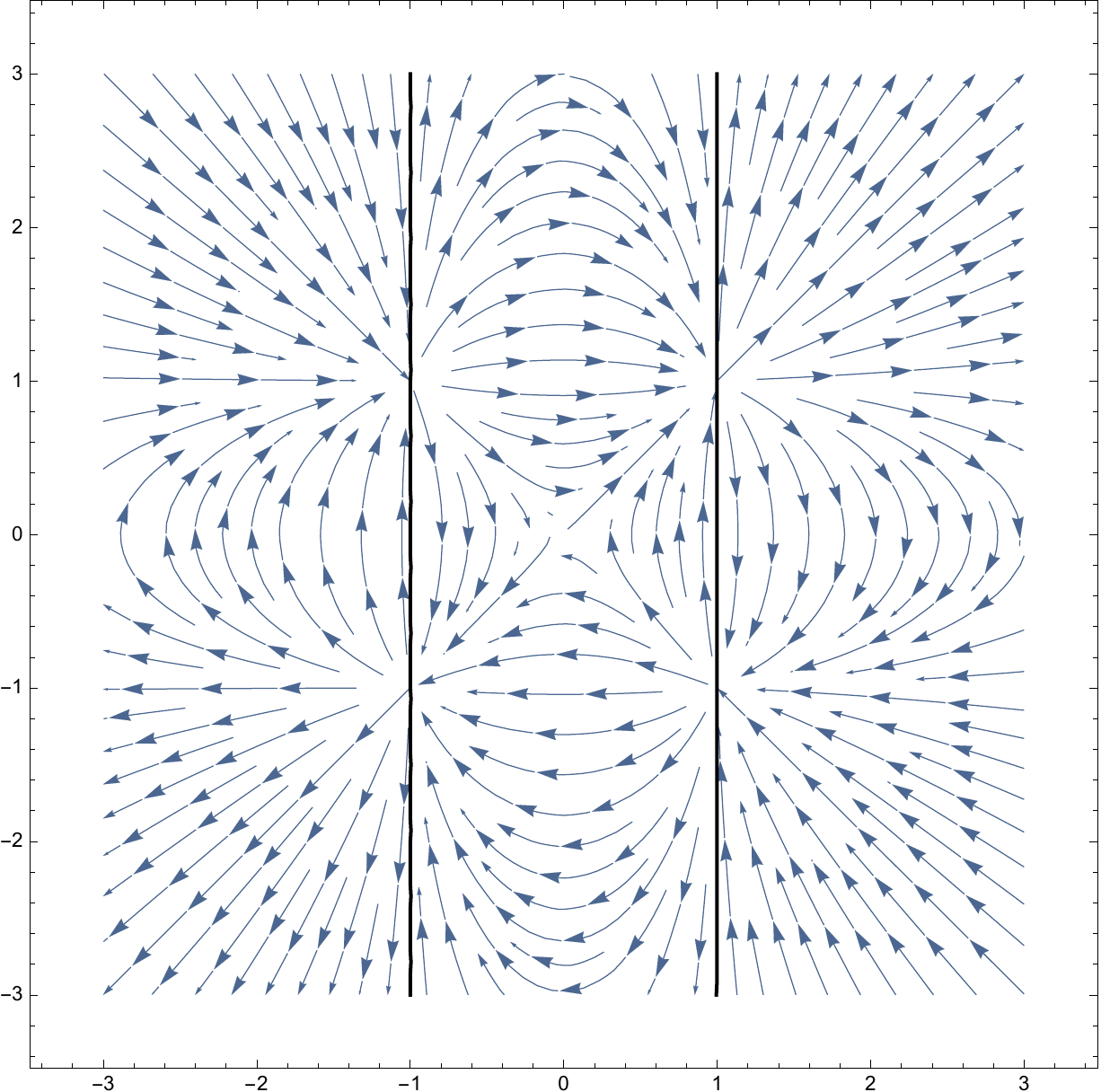}\hspace{0.5cm}\includegraphics[width=0.4\textwidth]{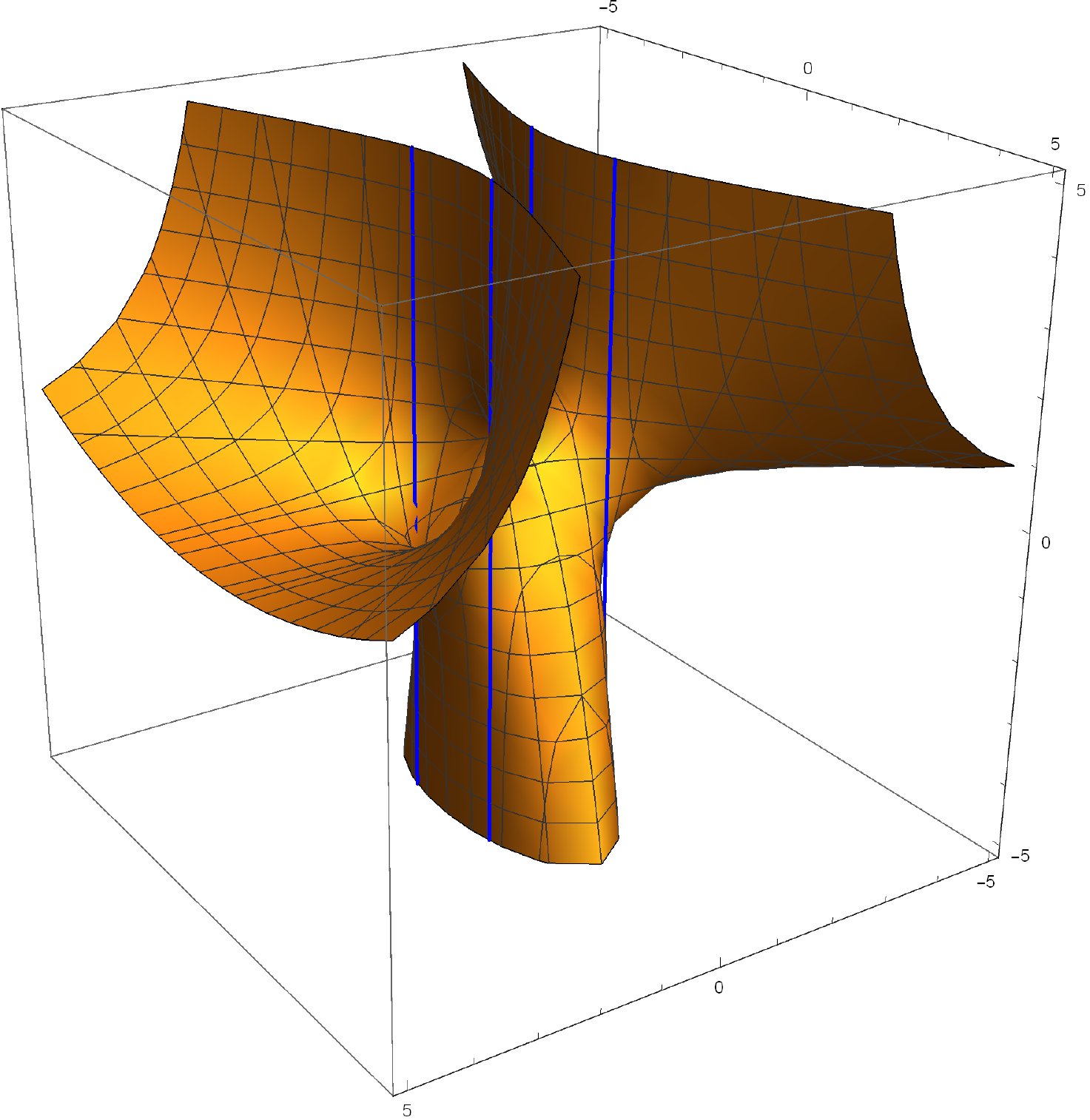}}
	\caption{\small{Phase portrait of   $X(x,y,z)=(2y,2xz,(x^{2} - 1)z - (y^{2} - 1))$ (left) and surface $M$ (right). 
			The pseudo impasse set $\mathcal{I}_{M}$ is denoted in blue.}}
\end{figure}

Systems written as \eqref{exe-intro-impasse} are known as \emph{constrained systems} (or \emph{impasse systems}).  
Constrained systems have been widely studied in the literature. In \cite{Zhito} the author classified normal forms 
 in $\mathbb{R}^{2}$ and in \cite{SotoZhito} the authors gave normal forms defined in $\mathbb{R}^{n}$, $n\geq 3$. 
 Both references assume that the impasse manifold is smooth.  Applications in electrical circuits can be found in \cite{Smale}.\\
 
Below we briefly list some results that we have proved about flows and impasses.

\begin{itemize}
	\item The flow of $X$ on the algebraic invariant manifold $M$  is determined by a constrained system 
	$A(x,y)(\dot{x},\dot{y}) = F(x,y),$ $ (x,y)\in\mathbb{R}^{2}$ (see section 2 for a precise definition). 
	The projection of $\mathcal{I}_{M}$ on $\mathbb{R}^{2}$ lies on the impasse set 
	$\mathcal{I}=\{(x,y):\det A(x,y)=0\}$ and it is a set of equilibrium points of the adjoint 
	vector field $A^{*}F$.  See \textit{Theorem \ref{teo-fluxo-impasse}}, section 3.
	
	\item Let $(x,y)\in\mathbb{R}^{2}$ be the projection of a line $r\subset\mathcal{I}_{M}$. Then 
	$(x,y)\in \mathcal{I}$ is a hyperbolic equilibrium point for the adjoint vector field $A^{*}F$  if, 
	and only if, $(x,y)$ is a hyperbolic node. Moreover, $r$ does not contain equilibrium points of 
	$X$, $r$ is transversal to the flow  and $r$ does not contain any singular points of the surface $M$. 
	See \textit{Proposition \ref{prop-singularidades-impasses}} and \textit{Theorem \ref{teo-fluxo-2-hiperbolico}}, 
	section 3.  If $(x,y)\in \mathcal{I}$ is a non hyperbolic equilibrium point then, under some conditions, $r$  
	intersects the singular part of $M$ or $r$ is invariant by the flow. 
	See \textit{Theorem \ref{teo-fluxo-nao-hip-2}}, section 3.
	
\end{itemize}

In section 4 we exemplify our results with well-known systems in the literature, for example, 
Falkner-Skan Equation, Lorenz System and Chen System. We discuss how the study of flows 
on invariant surfaces by means of a constrained system can be extended in higher dimensions.\\

In the second part of this paper we consider 1-parameter families of first integrals and unfoldings of 
minimal sets. In our approach, singular perturbation theory \cite{Fenichel} is the main tool. More 
precisely, we prove that equilibrium points and periodic orbits of  smooth system which are 
contained in  algebraic invariant surfaces persist under small perturbations.\\

\noindent \textbf{Example.} Let $H_{\e}:\mathbb{R}^{3}\rightarrow\mathbb{R}$ be a family 
of smooth functions given by $H_{\e}(x,y,z) = x - y^{2} - \e z$. For each $\varepsilon > 0$, $H_{\e}$ 
is a first integral of
$X_{\e}(x,y,z)=(2y(y + z) - \e(y - z), y + z,y - z).$ If $\varphi_{0}$ is a trajectory of $X_{0}$  
in the level $H_{0}(x,y,z) = 0$, then $\varphi_{0}$ is a solution of
\begin{equation}\label{exe-intro-spp2}
0 = x - y^{2},\ \dot{y} = y + z,\ \dot{z} = y - z.
\end{equation}

System \eqref{exe-intro-spp2} is an algebraic differential equation and it describes the 
slow flow on the slow manifold of a singular perturbation problem. Furthermore, 
\textit{Proposition \ref{main-theorem}}, section 5, says that $\varphi_{0}$ is a
 trajectory of $X_{0}$  if, and only if, $\varphi_{0}$ is a solution of \eqref{exe-intro-spp2}. 
 This allows us to study the flow of $X_{\e}$ for $\e > 0$. More precisely, we 
 study the flow of $X_{\e}$ in the levels $H_{\varepsilon}(x,y,z) = 0$ for $\e > 0$ using 
 singular perturbation problems. For this purpose, Fenichel Theorem is our main tool.
Concerning to 1--parameter families of smooth vector fields we prove:\\

\begin{itemize}
  \item Let $H_{\e}$ be a family of first integrals of  $n$-dimensional smooth vector fields $X_{\varepsilon}$. 
  The  equilibria of $X_{\e}$ on $\{H_{\e} = 0\}$ are equilibria of a singular perturbation problem. 
  If $p_{0}\in \{H_{0} = 0\}$ is an equilibrium point of $X_{0}$,
   then there exists a sequence of equilibrium points $p_{\e}$  of $X_{\varepsilon}$,  satisfying 
   that  $p_{\e}\rightarrow p_0$ and $p_{\varepsilon}\in \{H_{\varepsilon} = 0\}$. 
   See \textit{Proposition \ref{prop-spp-equilibrio}}, section 5.
  
  \item Under some conditions, periodic orbits of $X_{0}$ contained in $\{H_{0} = 0\}$
   persist under small perturbations. See \textit{Proposition \ref{prop-spp-periodico}}, section 5.
\end{itemize}

The paper is organized as follows. In section \ref{sec-definitions} we present basic concepts 
and definitions concerning constrained systems and singular perturbation theory. 
We start section \ref{sec-constrained} studying some geometric properties of the 
algebraic invariant surface, and then we present results that relate the flows on $M$ 
to constrained systems. We exemplify these results in section \ref{sec-examples} 
with well-known systems in the literature, such as Falkner-Skan, Lorenz and Chen systems. 
We also discuss how the problem of describe flows on invariant surfaces by means of 
constrained systems extends in higher dimensions. Finally, in section \ref{sec-spp} 
we deal with families of first integrals and slow-fast systems.

\section{Preliminaries on the Geometric  Singular Perturbation Theory and Impasses}\label{sec-definitions}

In this section we present basic concepts and definitions concerning constrained systems and singular 
perturbation theory. We also refer \cite{SotoZhito, Zhito} and \cite{Fenichel} for 
an introduction of constrained systems and singular perturbation theory, respectively.

\subsection{Constrained systems}

A \emph{constrained system} (or \emph{impasse system}) is given by
\begin{equation}\label{impasse}
A(\mathbf{x})\dot{\mathbf{x}} = F(\mathbf{x}),\quad \mathbf{x}\in\R^n
\end{equation}
where   $F:\R^n\rightarrow\R^n$ is a smooth vector field and $A(\mathbf{x})$ is a square matrix of order $n$ whose entries
$a_{ij}(\mathbf{x})$  smoothly depend on $\mathbf{x}$. This kind of system 
generalizes vector fields because at the points where $\det A(\mathbf{x}) \neq 0$ 
we can rewrite \eqref{impasse} as
$\dot{\mathbf{x}} = A^{-1}(\mathbf{x})F(\mathbf{x}).$
On the other hand, at the points where $\det A(\mathbf{x}) = 0$ (called \emph{impasse points}) 
we cannot assure the existence and/or uniqueness of the solutions.
Another particularity of \eqref{impasse} is the existence of the 
\emph{impasse manifold}, defined by
\begin{equation}\label{impasse-manifold}
\mathcal{I} = \{\mathbf{x}\in\mathbb{R}^{n} | \det A(\mathbf{x}) = 0\}.
\end{equation}

We can draw the phase portrait of \eqref{impasse} as follows. We relate the system \eqref{impasse} 
to the vector field $A^{*}F$, where $A^{*}$ is the {adjoint matrix of $A$} characterized by 
$AA^{*} = A^{*} A = \det (A)I$. Thus, the phase portrait of \eqref{impasse} can be seen as 
the phase portrait of the vector field $A^{*}F$ by removing the impasse points from its orbits 
and inverting its orientation where $\det A < 0$. As in \cite{CardinSilvaTeixeira}, the vector 
field $A^{*}F$ will be called {adjoint vector field}.

It can be found normal forms for constrained systems in $\mathbb{R}^{2}$ and $\mathbb{R}^{n}$ 
with $n\geq 3$ in \cite{Zhito} and \cite{SotoZhito}, respectively. In both references, 
the authors studied the dynamics of \eqref{impasse} in the neighborhood of a {regular point} 
$\mathbf{x}_{0}\in \mathcal{I}$, that is, $d\big{(}\det A(\mathbf{x})\big{)}(\mathbf{x}_{0}) \neq 0$. 
Moreover, they adopted the hypothesis that the trajectories of the system \eqref{impasse} 
either do not intercept $\mathcal{I}$, or intercept $\mathcal{I}$ in a finite number of isolated points.

Let $\mathbf{x}_{0}\in \mathcal{I}$ be a regular impasse point and consider the following conditions.
\begin{itemize}
	\item[(A)] The vector space $\ker A(\mathbf{x}_{0})$ is transversal to $\mathcal{I}$.
	\item[(B)] The vector $F(\mathbf{x}_{0})$ does not belong to the range of $A(\mathbf{x}_{0})$.
\end{itemize}

By linear algebra, we know that Im$A = \ker A^{*}$ and $\ker A = $ Im$A^{*}$. Therefore, 
condition {(A)} means that the vector field $A^{*}F$ is transversal to $\mathcal{I}$ at 
$\mathbf{x}_{0}$ and condition {(B)} means that $\mathbf{x}_{0}$ is not an equilibrium point of $A^{*}F$.

\begin{definition}\label{def-impasse-singularidades}
	Let $\mathbf{x}_{0}\in \mathcal{I}$ be a regular impasse point of \eqref{impasse}. 
	\begin{enumerate}
		\item $\mathbf{x}_{0}$ is non singular if $\mathbf{x}_{0}$ satisfies conditions 
		{(A)} and {(B)}.
		\item $\mathbf{x}_{0}$ is a K-singularity (kernel singularity) if $\mathbf{x}_{0}$ 
		satisfies {(B)} and does not satisfy {(A)}.
		\item $\mathbf{x}_{0}$ is an R-singularity (range singularity) if $\mathbf{x}_{0}$ 
		satisfies {(A)} and does not satisfy {(B)}.
		\item $\mathbf{x}_{0}$ is an RK-singularity (range-kernel singularity) if $\mathbf{x}_{0}$ 
		does not satisfies conditions {(A)} and {(B)}.
	\end{enumerate}
\end{definition}

\subsection{Slow--fast systems and Fenichel Theory}

A \emph{singularly perturbed system} is a system of the form
\begin{equation}\label{sis-slowfast}
\varepsilon\dot{\mathbf{x}} = f(\mathbf{x},\mathbf{y},\e), \ \dot{\mathbf{y}} = g(\mathbf{x},\mathbf{y},\e),
\end{equation}
where $f,g$ are smooth, $\mathbf{x}\in\mathbb{R}^{n}$, $\mathbf{y}\in\mathbb{R}^{m}$ and $\varepsilon > 0$ is small. 
The dot $\cdot$ denotes the derivative with respect to $t$.

The parameter $\varepsilon$ measures the variation rate of $\mathbf{x}$ and $\mathbf{y}$. When 
$\e = 0$, the system (\ref{sis-slowfast}) reduces to the differential-algebraic system
\begin{equation}\label{sis-lento}
0 = f(\mathbf{x},\mathbf{y},0), \ \dot{\mathbf{y}} = g(\mathbf{x},\mathbf{y},0).
\end{equation}

System (\ref{sis-lento}) is called \emph{reduced problem} or \emph{slow equation}. By taking 
$t = \e\tau$ in (\ref{sis-slowfast}), we obtain the system
\begin{equation}\label{sis-slowfast2}
\mathbf{x}' =  f(\mathbf{x},\mathbf{y},\e), \ \mathbf{y}' =  \e g(\mathbf{x},\mathbf{y},\e),
\end{equation}
where $'$ denotes the derivative with respect to $\tau$. In the limit $\e = 0$, we obtain the
 \emph{fast equation} (or \emph{layer problem}) given by
\begin{equation}\label{sis-rapido}
\mathbf{x}' =  f(\mathbf{x},\mathbf{y},0), \ \mathbf{y}' =  0.
\end{equation}

System \eqref{sis-rapido} is reduced to $n$ differential equation with respect to the fast 
variable $\mathbf{x}$, which depends on the slow variable $\mathbf{y}$ as a parameter. For $\e>0$, 
systems (\ref{sis-slowfast}) and (\ref{sis-slowfast2}) are equivalents.

The set
\begin{equation}\label{slowmani}
S_{0} = \{(\mathbf{x},\mathbf{y})\in\mathbb{R}^{n}\times\mathbb{R}^{m}|f(\mathbf{x},\mathbf{y},0) = 0\}
\end{equation}
is the phase space of (\ref{sis-lento}) and it is known in the literature as \emph{critical manifold} 
or \emph{slow manifold}. Notice that $S_{0}$ is the set of equilibrium points of (\ref{sis-rapido}).

We can interpret the phase portrait (\ref{sis-slowfast}) and (\ref{sis-slowfast2}) when $\e$ is 
close to $0$ as follows. A point outside $S_{0}$ moves from a stable fast fiber according to the 
dynamics of (\ref{sis-rapido}), until it reaches a stable branch of $S_{0}$. Then, the dynamics 
change to (\ref{sis-lento}). If the corresponding solution reaches a singularity or a bifurcation 
point (where $S_{0}$ loses stability), thus the dynamics changes to (\ref{sis-rapido}).\\

A point $p = (\mathbf{x},\mathbf{y})\in S_{0}$ is \emph{normally hyperbolic} if 
$f_\mathbf{x} (p, 0)$ is a matrix whose all 
eigenvalues have nonzero real parts being $k^s$ eigenvalues with negative real parts and $k^u$ eigenvalues with 
positive real parts. The set of all normally hyperbolic points of $S_{0}$ is denoted by $\mathcal{NH}(S_{0})$.\\

The following theorem is one of the most important results of singular perturbation theory and 
it is due to Fenichel. 
Such result describes how is the flow of system \eqref{sis-slowfast} for $\e$ is sufficiently small. 
See \cite{Fenichel} for details.

\begin{theorem}\label{teo-Fenichel}
Let $N_{0}\subset \mathcal{NH}(S_{0})$ be a $j$-dimensional normally hyperbolic compact 
submanifold of $S_{0}$ for \eqref{sis-lento}, with a $(j + j^{s})$-dimensional local stable 
manifold $W^{s}$ and a $(j + j^{u})$-dimensional local unstable manifold $W^{u}$. Suppose 
that system \eqref{sis-slowfast} satisfies $f,g\in C^{r}$. Then for $\varepsilon > 0$ 
sufficiently small the following statements are true.
	\begin{itemize}
		\item[F1.] There exists a $C^{r-1}$ family of compact locally invariant manifolds 
		$\{N_{\varepsilon}\}_{\varepsilon}$ of \eqref{sis-slowfast} converging to $N_{0}$, 
		according Hausdorff distance. Moreover, $N_{\varepsilon}$ is diffeomorphic to $N_{0}$.
		\item[F2.] There exist $C^{r-1}$ families of $(j +j^s + k^s)$-dimensional and 
		$(j +j^u +k^u)$-dimensional manifolds $N^{s}_{\varepsilon}$ and $N^{u}_{\varepsilon}$, 
		respectively, such that $N^{s}_{\varepsilon}$ and $N^{u}_{\varepsilon}$ are 
		the local stable and unstable manifolds of $N_{\varepsilon}$.
	\end{itemize}
\end{theorem}

\subsection{Catastrophes as invariant surfaces}

The slow flow \eqref{sis-lento} of a singular perturbation problem, with $(x,y,z)\in\R^3$, 
is given by a constrained system
\begin{equation}\label{potential-equation}
0 =  V_x(x,y,z), \ \dot{y} = \beta(x,y,z), \ \dot{z} = \gamma(x,y,z),
\end{equation}
where $V$ is the \emph{potential function}. According Takens \cite{Takens}, under
 topological equivalence, there are 12 normal forms of generic constrained differential
equations. 
  \begin{center}
  \begin{tiny}
  \begin{tabular}{|c|c|c|}
    \hline
    $V(x,y,z)$ & $X(x,y,z)$ & Type\\
    \hline
    $\frac{x^{2}}{2}$ 
     & $ (0,1,0)      $ & Flow-box \\
     & $ (0,y,z)$   & Source \\
     & $  (0,y,-z)$ & Saddle \\
     & $       (0,-y,-z)$ & Sink \\
    \hline
     $\frac{x^{3}}{3} + yx$ 
      & $        (0,1,0)$ & Flow-box 1\\
      & $      (0,-1,0)$ & Flow-box 2\\
      & $       (0,3x + z,1)$ & Source \\
      & $        (0,-3x + z,1)$ & Sink \\
      & $         (0,-z,1)$ & Saddle \\
      & $         (0,x + z,1)$ & Focus \\
     \hline
    $\frac{x^{4}}{4} + \frac{zx^{2}}{2} + yx$ 
      & $       (0,1,0)$ & Flow-box 1 \\
      & $         (0,-1,0)$ & Flow-box 2 \\
     \hline
  \end{tabular}
  \end{tiny}
  \end{center}

Following Proposition, whose proof can be found in \cite{BKK}, says the slow flow on the slow manifold is given by 
a constrained system.

\begin{proposition} Consider  system \eqref{potential-equation} 
and  a point $p$ satisfying
\begin{itemize}
  \item[(a)] $V_x(p) = 0$, $V_{xx}(p) = 0$;
  \item[(b)] $V_{xy}(p) \neq 0$ or $V_{xz}(p) \neq 0$.
\end{itemize}
Then system \eqref{potential-equation} can be written as
\begin{equation}\label{dessing}
-V_{xx}\dot{x} = V_{xy}\beta + V_{xz}\gamma, \ \dot{z} = \gamma,
\end{equation}
where $y = \xi(x,z)$ is the solution of $V_x(p) = 0$.
\end{proposition}

If $V(x,y,z) = \frac{x^{2}}{2}$, then equation \eqref{dessing} defines a smooth system.\\

 If $V(x,y,z) = \frac{x^{3}}{3} + yx$, then equation \eqref{dessing} defines a 
constrained system where $\mathcal{I} = \{(x,z)\in\mathbb{R}^{2}|x = 0\}$. 
For the Flow-Box normal forms, system \eqref{dessing} is given by
$-2x\dot{x} = \mp 1, \ \dot{z} = 0$, where the adjoint vector field is
$\dot{x} = \mp 1, \ \dot{z} = 0.$ It follows that the origin is non singular. 
The source and sink normal forms are
$-2x\dot{x} = \pm 3x + z, \ \dot{z} = 1,$
where the adjoint vector field is
$\dot{x} = \pm 3x + z, \ \dot{z} = -2x.$ Since the origin is a node for the adjoint vector field,
 it is a R-singularity. The same occurs for the saddle and focus normal forms. 
For the saddle normal form, system \eqref{dessing} takes form
$-2x\dot{x} = -z, \ \dot{z} = 1,$
where the adjoint vector field is
$\dot{x} = -z, \ \dot{z} = -2x.$
For the focus normal form, system \eqref{dessing} takes form
$
-2x\dot{x} = x+z, \ \dot{z} = 1,
$
where the adjoint vector field is
$
\dot{x} = x+z, \ \dot{z} = -2x.$\\

If $V(x,y,z) = \frac{x^{4}}{4} + \frac{zx^{2}}{2} + yx$, then the impasse set 
of \eqref{dessing} is given by $\mathcal{I} = \{(x,z)\in\mathbb{R}^{2}|3x^{2} + z = 0\}$
 and the flow-box normal forms are given by
$
-(3x^{2} + z)\dot{x} = 1, \ \dot{z} = 0,
$
where the adjoint vector field is
$\dot{x} = 1, \ \dot{z} = 0.$
In this case, the origin is a K-singularity.

\begin{figure}[h!]\label{fig-catastrofe-no}
	\center{\includegraphics[width=0.3\textwidth]{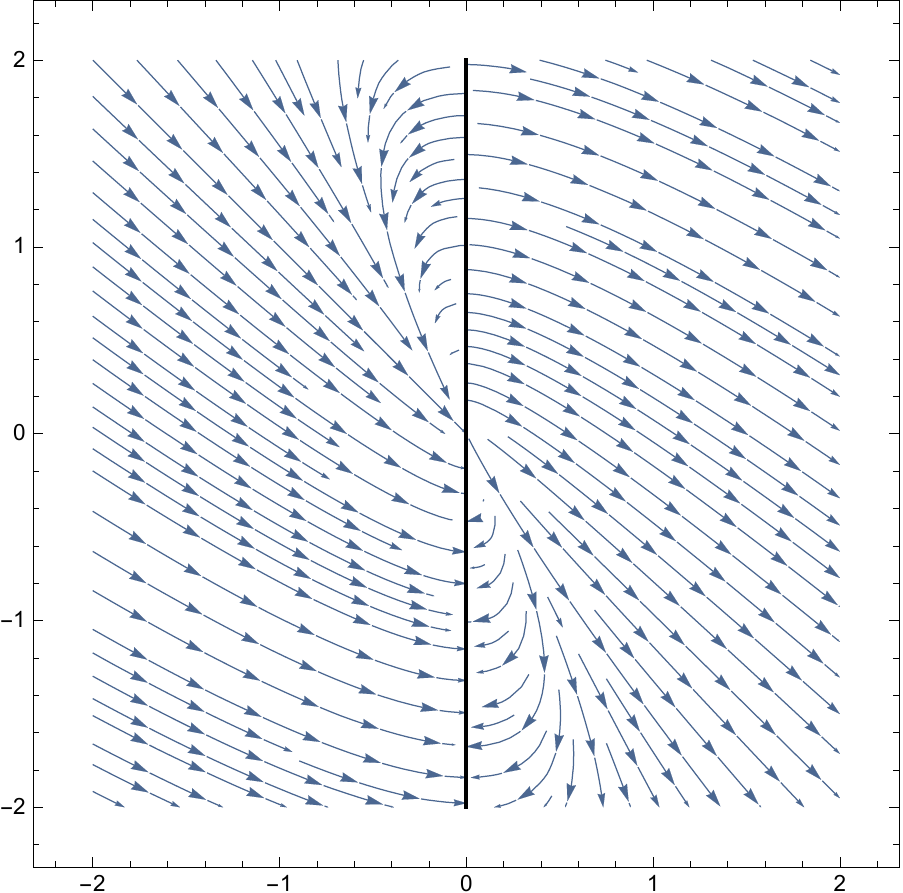}
		\hspace{0.5cm}\includegraphics[width=0.3\textwidth]{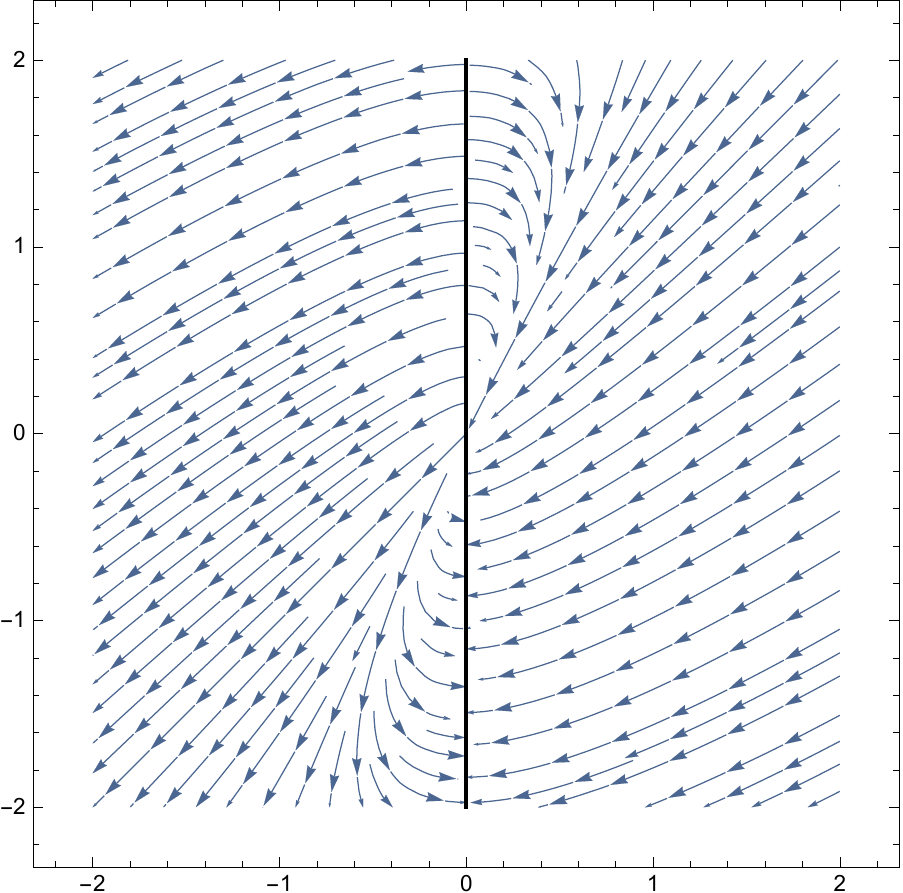}}
	\caption{\small{Phase portraits of \eqref{dessing} for the source (left) and sink (right) cases}.}
\end{figure}

\begin{figure}[h!]\label{fig-catastrofe-sela-foco}
	\center{\includegraphics[width=0.3\textwidth]{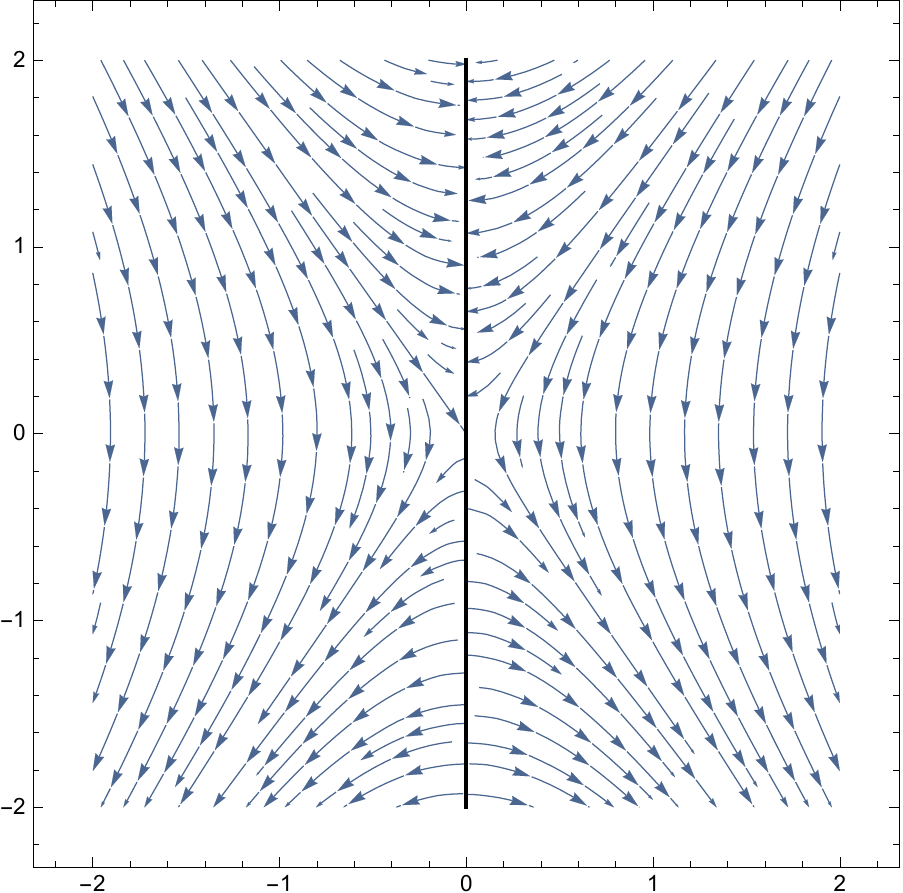}\hspace{0.5cm}
		\includegraphics[width=0.3\textwidth]{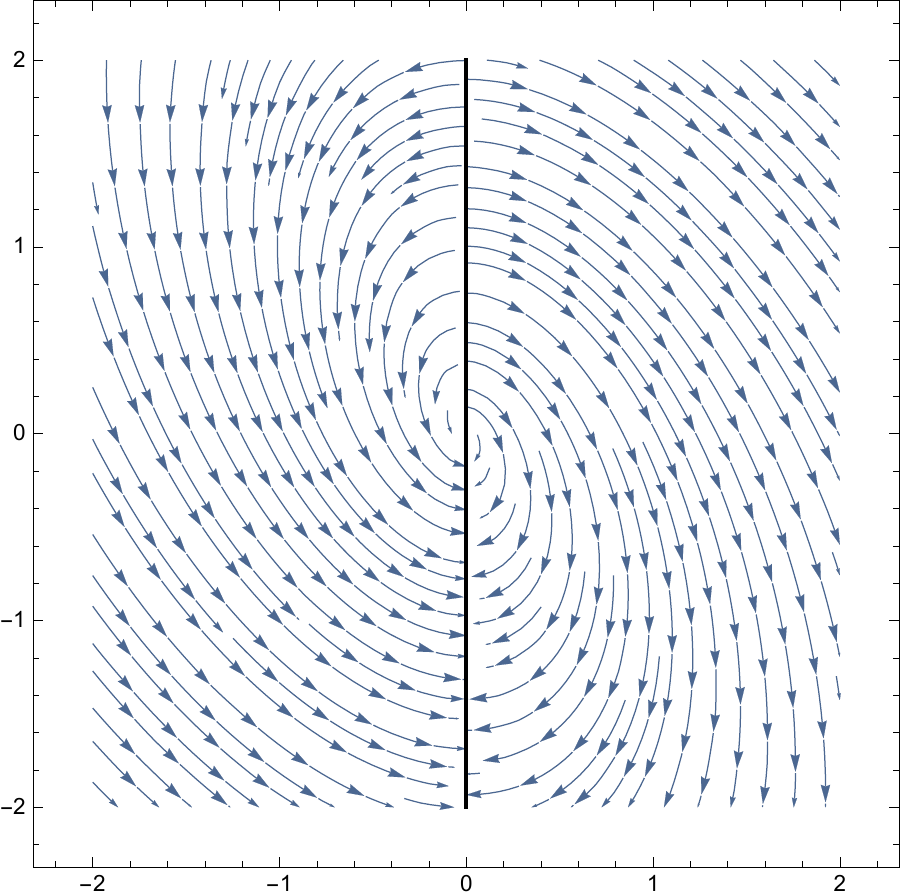}}
	\caption{\small{Phase portraits of \eqref{dessing} for the saddle (left) and focus (right) cases}.}
\end{figure}

\begin{figure}[h!]\label{fig-catastrofe-dobra}
	\center{\includegraphics[width=0.3\textwidth]{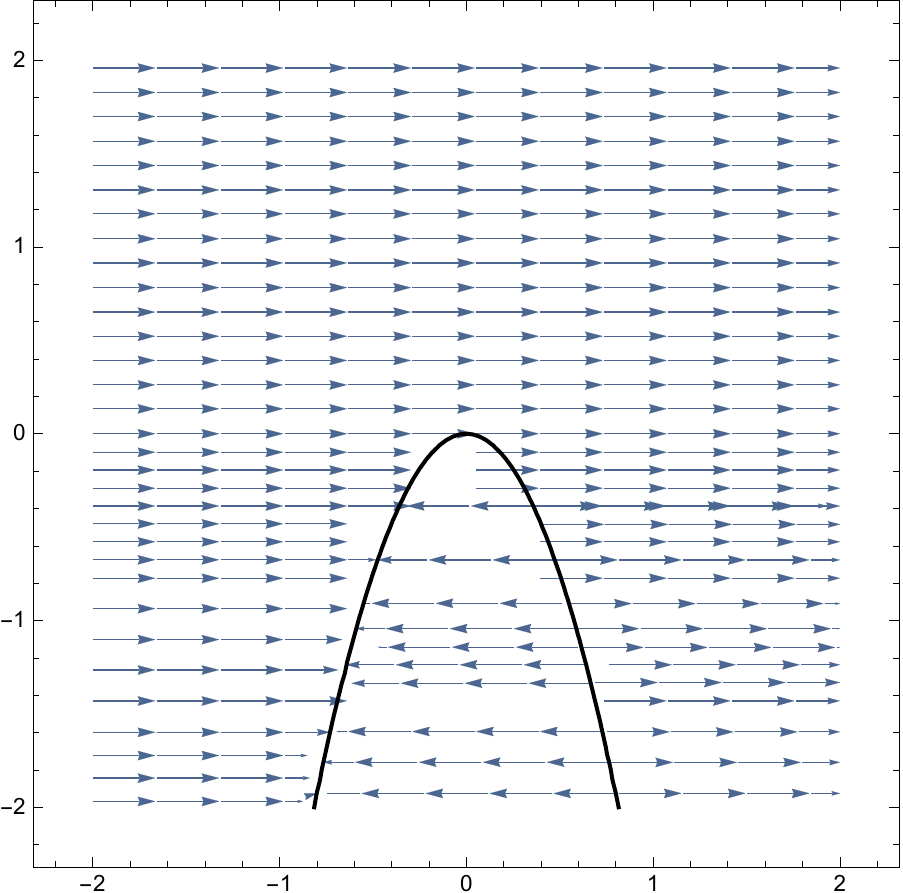}}
	\caption{\small{Phase portraits of \eqref{dessing}, fold case.}}
\end{figure}

\section{Phase portrait on invariant algebraic surfaces}\label{sec-constrained}
Let $X$ be a smooth vector field in $\R^3$ with polynomial first integral
$H(x,y,z) = f(x,y)z - g(x,y).$  Denote 
$M=H^{-1}(0) = \mathcal{G}_{M}\cup\mathcal{I}_{M}$
where  $\mathcal{G}_{M}$ is the  graphic of $z=(g/f)(x,y)$ and 
 $\mathcal{I}_{M}\subset M$ is the \textit{pseudo-impasse} subset of $M$.
 The next proposition summarizes some properties of $M$.

\begin{proposition}\label{prop-superficie} Let $\mathcal{Z}_f$ and $\mathcal{Z}_g$ be the zero sets of $f$ and $g$, respectively.
	\begin{itemize}
		\item[(a)] If $\mathcal{Z}_f\cap\mathcal{Z}_g = \emptyset$, then $\mathcal{I}_{M} =
		 \emptyset$ and $M = \mathcal{G}_{M}$.
		\item[(b)] If $\mathcal{Z}_f\cap\mathcal{Z}_g$ is a set containing isolated points, then 
		$\mathcal{I}_{M}$ is a set formed by lines which are parallel to the $z$-axis. See figure \ref{conjunto-IM}.
		\item[(c)] If $\mathcal{Z}_f$ and $\mathcal{Z}_g$ coincide in an open set of $\mathbb{R}^{2}$, 
		then $\mathcal{I}_{M}$ is a surface that is ortogonal to the $xy$-plane.
		\item[(d)] If there exist $p\neq q\in\R^2$ such that $f(p).f(q)<0$ and $\mathcal{I}_{M}=\emptyset$ 
		then $M$ is disconnected.
		\item[(e)] If $p$ is a singular point of $M$, then $p\in\mathcal{I}_{M}$.
	\end{itemize}
\end{proposition}
\begin{proof}
	We can see the sets $\mathcal{Z}_f$ and $\mathcal{Z}_g$ as  curves in $\mathbb{R}^{2}$. 
	This implies that $\mathcal{I}_{M}$ is a set of lines such that they are parallel to the $z$-axis 
	and they intercept the $xy$-plane at the points of intersection of the curves $\mathcal{Z}_f$ 
	and $\mathcal{Z}_g$. Then it follows that the itens (a),(b) and (c) are true. For item (d), suppose 
	that $\mathcal{I}_{M} = \emptyset$, that is, $M = \mathcal{G}_{M}$. By hypothesis, $f$ 
	assumes positives and negatives values, therefore we can write $M = \mathcal{G}^{+}_{M}\cup\mathcal{G}^{-}_{M}$, 
	where $\mathcal{G}^{+}_{M} = \{f(x,y) > 0, z = (g/f)(x,y)\}$ 
	and $\mathcal{G}^{-}_{M} = \{f(x,y) < 0, z = (g/f)(x,y)\}$. 
	The sets $\mathcal{G}^{+}_{M}$ and $\mathcal{G}^{-}_{M}$ are open sets
	 contained in $M$ such that $\mathcal{G}^{+}_{M}\cap\mathcal{G}^{-}_{M} = \emptyset$, 
	 thus the surface is disconnected.
	Now note that a point $p = (x_{0},y_{0},z_{0})\in M$ is singular if, and only if, 
	$\nabla H(p) = 0$. In other words, $p$ will be a singular point of $M$ if
	$$
	f_{x}(x_{0},y_{0})z_{0} = g_{x}(x_{0},y_{0}), \ f_{y}(x_{0},y_{0})z_{0} = g_{y}(x_{0},y_{0}), \ f(x_{0},y_{0}) = 0.
	$$
	The last equation assures that $p\in \mathcal{I}_{M}$ and thus item (e) holds.
\end{proof}

\begin{figure}[h!]
	\center{\includegraphics[width=0.4\textwidth]{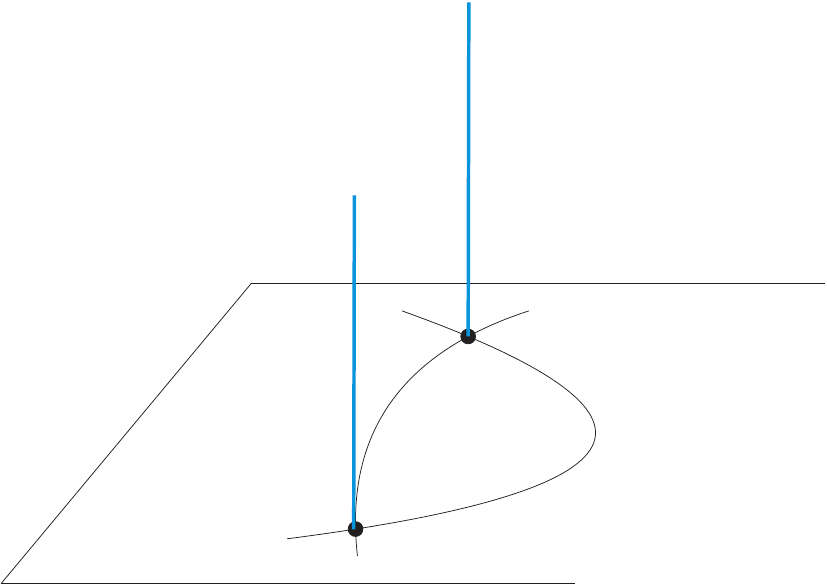}}\\
	\caption{\small{Pseudo-impasse set $\mathcal{I}_{M}$ in blue.}}\label{conjunto-IM}
\end{figure}

The set $\mathcal{I}_{M}$ is not necessarily the set of all singular points of $M$. 
As we shall see, there are examples of regular surfaces such that $\mathcal{I}_{M}\neq \emptyset$. 
Moreover, the hypothesis that $f$ assumes positive and negative values is crucial 
in Proposition \ref{prop-superficie}.\\

In the particular case where $X$ is a polynomial vector field
we can $\dot{\mathbf{x}}=X(\mathbf{x})$ to an analytic system on a 
closed ball of radius one, whose interior is diffeomorphic to $\R^3$ and its  boundary,
the $2$--dimensional sphere $\s^2,$ plays the role of the infinity.
This closed ball is denoted by $\D^3$ and called the
\emph{Poincar\'{e} ball}, because the technique for doing such an
extension is precisely the {\it Poincar\'{e} compactification} for a
polynomial differential system  in $\R^3$, which is described in
details in \cite{CL}. Besides, we also can extend $M$ to the infinity.  
In \cite{LlibreMessiasSilva} the authors proved the following Lemma in 
order to obtain the expression of a surface at infinity.

\begin{lemma}\label{lemma-paulo}
	Let $H:\mathbb{R}^{3}\rightarrow\mathbb{R}$ be a polynomial  of degree 
	$m$ and $H(x,y,z) = 0$ be an algebraic surface.
	 The extension of this surface to the boundary of the Poincar\'{e} ball 
	 $x^2+y^2+z^2=1$ is obtained solving the system
	$w^{m}H(x/w,y/w,z/w) = 0$, $w = 0$.
	\end{lemma}

If $P$ is a polynomial, denote the degree of $P$ by $\mathrm{d}(P)$. We can write $P$ as
$P(x,y) =\sum_{j = 0}^{m}P_{j}(x,y),$
where $P_{j}$ is a homogeneous polynomial of degree $j$.

\begin{proposition}\label{lemma-infinity}
	The extension of the surface $M$ to infinity contains 
	the poles $(0,0,\pm 1)$. In particular, if $\mathrm{d}(f) \geq \mathrm{d}(g)$ 
	then such extensions contain the big circle $\{z = 0\}$.
\end{proposition}
\begin{proof}
	If $H$ is a polynomial of degree $m$ we have three cases to consider.
	\begin{itemize}
		\item $\mathrm{d}(g) = m$ and $\mathrm{d}(f) < m-1$: Lemma \ref{lemma-paulo} provides $g_{m}(x,y) = 0$ and therefore the
		extension of $M$ is the same as the extension of the surface $\{g_{m}(x,y) = 0, z\in\mathbb{R}\}$.
		\item $\mathrm{d}(g) < m$ and $\mathrm{d}(f) = m-1$: Lemma \ref{lemma-paulo} provides $f_{m-1}(x,y)z = 0$, and this implies
		that the extension of $M$ is the union of the great circle $\{z = 0\}$ with the extensions of the
		surface $f_{m-1}(x,y) = 0$.
		\item $\mathrm{d}(g) = m$ and $\mathrm{d}(f) = m-1$: Lemma \ref{lemma-paulo} provides $f_{m-1}(x,y)z - g_{m}(x,y) = 0$.
	\end{itemize}
	Since $g_{m}(x,y)$ and $f_{m-1}(x,y)$ are homogeneous polynomials of degree $m$ and $m-1$
	respectively,  the poles $(0,0,\pm 1)$ belongs to the projection at the infinity.
\end{proof}

Let $X(x,y,z)=(\alpha(x,y,z), \beta(x,y,z),  \gamma(x,y,z)),$
be a smooth vector field with $\alpha, \beta$ given by
$\alpha(x,y,z) = \sum_{i = 0}^{m}\alpha_{i}(x,y)z^{i}$, $\beta(x,y,z) = \sum_{j = 0}^{n}\beta_{j}(x,y)z^{j}.$

\begin{theorem}\label{teo-fluxo-impasse}
If $M = \{f(x,y)z - g(x,y) = 0\}$ is an invariant algebraic surface of $X=(\alpha, \beta,  \gamma),$ then there exists
 a constrained system defined in $\mathbb{R}^{2}$  such that:
	\begin{enumerate}
		\item The impasse curve of such system is given by $\mathcal{I} = \{(x,y) | f(x,y) = 0\}$.
		\item On $\mathbb{R}^{2}\backslash\mathcal{I}$, the orbits of the constrained system are  the 
		projections of the ones on  $\mathcal{G}_{M}$.
		\item The projection of $\mathcal{I}_{M}$ on $\mathbb{R}^{2}$ is contained in $\mathcal{I}$ 
		and it is a set of equilibrium points of the adjoint system.
	\end{enumerate}
\end{theorem}
\begin{proof}
	
	If $(x,y,z)\in\mathcal{G}_{M}$ then $z = \xi(x,y) = (g/f)(x,y)$. Thus the orbits of $X$ on $\mathcal{G}_{M}$ are
	 obtained from the solutions of
	\begin{equation}\label{sis-impasse-geral1}
	\dot{x}  =  \sum_{i = 0}^{m}\alpha_{i}(x,y)\Big{(}(g/f)(x,y)\Big{)}^{i}, \quad
	\dot{y}  = \sum_{j = 0}^{n}\beta_{j}(x,y)\Big{(}(g/f)(x,y)\Big{)}^{j},\end{equation}
	
Rewriting \eqref{sis-impasse-geral1} we get the constrained system
	\begin{equation}\label{sis-impasse-geral2}
	f^{m}\dot{x}  =  \sum_{i = 0}^{m}f^{m-i}g^{i}\alpha_{i}, \quad
	f^{n}\dot{y} = \sum_{j = 0}^{n}f^{n-j}g^{j}\beta_{j},
	\end{equation}
	where $f^{k}(x,y) = \Big{(}f(x,y)\Big{)}^{k}$ and $g^{l}(x,y) = \Big{(}g(x,y)\Big{)}^{l}$.
	
	The impasse curve of \eqref{sis-impasse-geral2} is  $\mathcal{I}=\{(x,y)|f(x,y) = 0\}$ and  
	system \eqref{sis-impasse-geral2} has the phase portrait of system \eqref{sis-impasse-geral1} 
	on the region $f(x,y) \neq 0$. Therefore, the orbits of $X$ on $\mathcal{G}_{M}$ are the image
	by $\xi$ of the orbits of 	\eqref{sis-impasse-geral2} on $\mathbb{R}^{2}\backslash\mathcal{I}$.
	
	The adjoint system of \eqref{sis-impasse-geral2} is given by
	\begin{equation}\label{sis-impasse-regularizado}
	\dot{x} =  f^{n}\Bigg{(}\sum_{i = 0}^{m}f^{m-i}g^{i}\alpha_{i}\Bigg{)}, \quad
	\dot{y} = f^{m}\Bigg{(}\sum_{j = 0}^{n}f^{n-j}g^{j}\beta_{j}\Bigg{)}.
	\end{equation}
	
	Moreover, the projection of $\mathcal{I}_{M}$ on $\mathbb{R}^{2}$ is given by
	$\{(x,y)| f(x,y) = 0 = g(x,y)\},$ which is a set containing equilibrium points of the 
	adjoint vector field \eqref{sis-impasse-regularizado}.
\end{proof}

Theorem \ref{teo-fluxo-impasse} shows us that outside the impasse curve $\mathcal{I}$ 
the orbits of system \eqref{sis-impasse-geral2} are the orbits
of $X$ on $\mathcal{G}_{M}$. Thus, our objective is to describe the orbits of $X$ by 
points on $\mathcal{I}_{M}$.

In the proof of Theorem \ref{teo-fluxo-impasse}, observe that $\mathcal{I}$ is a curve 
of equilibrium points for the adjoint system \eqref{sis-impasse-regularizado}. 
This leads us to start our studies with less degenerate cases. For this, we will 
consider a system in $\mathbb{R}^{3}$ given by

\begin{equation}\label{sis-impasse-casoparticular-r3}
\dot{x} = \alpha(x,y), \ \dot{y} = \sum_{j = 0}^{n}\beta_{j}(x,y)z^{j}, \ \dot{z} = \gamma(x,y,z),
\end{equation}
which has invariant algebraic surface $M=H^{-1}(0)$, $H(x,y,x)=f(x,y)z-g(x,y).$
As in the proof of Theorem \ref{teo-fluxo-impasse} we obtain the constrained system
\begin{equation}\label{sis-impasse-casoparticular}
\dot{x} =  \alpha, \quad
f^{n}\dot{y} = g^{n}\beta_{n} + \sum_{j = 0}^{n-1}f^{n-j}g^{j}\beta_{j},
\end{equation}
whose adjoint system is given by
\begin{equation}\label{sis-impasse-casoparticular-regularizado}
\dot{x} = f^{n}\alpha, \quad
\dot{y}  =  g^{n}\beta_{n} + \sum_{j = 0}^{n-1}f^{n-j}g^{j}\beta_{j}.
\end{equation}

We already know by Theorem \ref{teo-fluxo-impasse} that the projection 
$\Pi\big{(}\mathcal{I}_{M}\big{)} \subset \mathcal{I}$
is a set containing only equilibrium points of the adjoint system. 
These are the points that we must study to understand the flows 
of \eqref{sis-impasse-casoparticular-r3} on $\mathcal{I}_{M}\subset M$.

The next result classifies the impasse points by means of the functions $f$ and $g$.

\begin{proposition}\label{prop-pontos-de-impasse}
	Let $p\in\R^2$ be an  impasse point of system \eqref{sis-impasse-casoparticular}.
	\begin{enumerate}
		\item $p$ is a {R-singularity} if and only if $g(p) = 0$ or $\beta_{n}(p) = 0$.
		\item $p$ is a {K-singularity} if and only if  $f_{y}(p) = 0$.
		\item $p$ is a {RK-singularity} if and only if $p$ is R-singularity and K-singularity simultaneously.
		\item $p$ is {non singular} if and only if $p$ does not satisfy any of the previous conditions.
	\end{enumerate}
\end{proposition}
\begin{proof}
	It follows directly from Definition \ref{def-impasse-singularidades} and
	 the expression of the adjoint system \eqref{sis-impasse-casoparticular-regularizado}.
\end{proof}

We  say that an impasse point of system \eqref{sis-impasse-casoparticular} $p$
 is  \textit{R-singularity of first kind} if $g(p) = 0$ and
 \textit{R-singularity of second kind} if $\beta_{n}(p) = 0$. The set of R-singularities of first
 kind is exactly the projection of $\mathcal{I}_{M}$ on the $xy$-plane. Moreover, it may appear R-singularities of first and second kind simultaneously.

\begin{corollary}\label{coro-superficie}
	Consider the adjoint system \eqref{sis-impasse-casoparticular-regularizado} 
	and the surface  $M=H^{-1}(0)$, $H(x,y,x)=f(x,y)z-g(x,y).$
	\begin{enumerate}
		\item Let $p$ be a singular point of $M$. Thus the projection of $p$ on $xy$-plane is a R-singularity
		of first kind for the adjoint system.
		\item If there exists $p\neq q\in\R^2$ such that $f(p)f(q)<0$  and the impasse curve $\mathcal{I}$
		does not have R-singularities of first kind then $M$ is disconnected.
	\end{enumerate}
\end{corollary}
\begin{proof}
	For item 1, recall that Proposition \ref{prop-superficie} assures that $p\in\mathcal{I}_{M}$.
	 It follows from the last remark that the projection of $p$ is a R-singularity of first kind. 
	 
	If $\mathcal{I}$ does not have R-singularities of first kind, then $\mathcal{I}_{M}$ is empty. By Proposition \ref{prop-superficie}, the item 2 is true.
\end{proof}

Let $r$ be a parallel line with respect to the $z$-axis and $q\in\mathcal{I}$ the point in which $r$ 
intercepts the $xz$-plane. From Corollary \ref{coro-superficie} we know that if $q$ is not a R or 
RK-singularity of first kind, then $r$ does not intercept the surface $M$. On the other hand, 
as a consequence of Proposition \ref{lemma-infinity} we have that $r$ intersects $M$ at infinity.

\begin{corollary}
	If $q\in\mathcal{I}$ is an isolated K-singularity, R-singularity of second kind or RK-singularity 
	of second kind, then the line $r = \{(p,z) | z\in\mathbb{R}\}$ intersects $M$ at the infinity.
\end{corollary}
\begin{proof}
	We already know by Proposition \ref{lemma-infinity} that the extension of $M$ at the infinity 
	contains the poles $(0,0\pm 1)$. The corollary follows observing that the line reaches such poles at the infinity.
\end{proof}

\noindent\textbf{Remark.}  We say that $H(x,y,z)$  is a
\emph{Darboux polynomial} of $X$ if it
satisfies $(\nabla H)\cdot X=\mu H,$ where $\mu=\mu(x,y,z)$  is a
real polynomial  called the {\it cofactor} of
$H(x,y,z)$. The surface $H(x,y,z)=0$ is an {\it invariant algebraic
surface}.\\

\noindent{\textbf{Example.}	The polynomial $H(x,y,z) = yz - 1$ is a Darboux polynomial of $X(x,yz)=(0,z,-z^3)$
	with cofactor $\mu(x,y,z) = -z^{2}$. The flows on $M = H^{-1}(0)$ are described by the constrained system
	$	\dot{x} = 0, \ y\dot{y} = 1.	$ All points on $\mathcal{I} = \{y = 0\}$ are non singular.
	 It follows from Corollary \ref{coro-superficie} that $M$ is disconnected. See Figure \ref{fig-exemplo-naosing}.\\

\noindent{\textbf{Example.}	 The polynomial $H(x,y,z) = (x + y^{2})z - 1$ is Darboux polynomial for $X(x,yz)=(0,z,-2yz^3)$
		with cofactor $\mu(x,y,z) = -2yz^{2}$. The flows on $M = H^{-1}(0)$ are described by the constrained system
	$\dot{x} = 0, \ (x + y^{2})\dot{y} = 1.$	The origin is a K-singularity. Moreover, it follows from 
	Corollary \ref{coro-superficie} that $M$ is disconnected. See Figure \ref{fig-exemplo-dobra}.\\

\noindent{\textbf{Example.}	 Let $\lambda\in\mathbb{R}$ be a parameter satisfying $|\lambda| > 1$. 
	The polynomial $H(x,y,z) = (y-x)z - 1$ is Darboux polynomial for
	$X(x,yz)=(1,\lambda yz,z^2-\lambda yz^3)$ with cofactor $\mu(x,y,z) = z -\lambda yz^{2}$. 
	The flows on $M = H^{-1}(0)$ are described by the constrained system
	$	\dot{x} = 1, \ (y - x)\dot{y} = \lambda y.	$ The origin is a hyperbolic equilibrium
	 point for the adjoint vector field, whose eigenvalues are $-1$ and $\lambda$ and 
	 its eigenvectors are $v_{-1} = (1,0)$ and $v_{\lambda} = (1,1 + \lambda)$. 
	 If $\lambda > 1$ the origin will be a saddle point and if $\lambda < 1$ the origin 
	 will be a node. In both cases, the origin is a R-singularity of second kind. 
	 See Figure \ref{fig-exemplo-singtipo2-sela}.\\

\noindent{\textbf{Example.}	
	Let $\lambda\in\mathbb{R}$ be a parameter satisfying $\lambda > 0$. 
	The polynomial $H(x,y,z) = (y-x)z - 1$ is Darboux polynomial for
	$X(x,y,z)=(\frac{1 + \lambda^{2}}{\lambda},-(\lambda x+y)z, (\lambda x+y)z^{3})$
	 with cofactor $\mu(x,y,z) = (\lambda x+y)z^{2}$. The flows on $M = H^{-1}(0)$ 
	 are described by the constrained system
	$
	\dot{x} = \frac{1 + \lambda^{2}}{\lambda}, \ y\dot{y} = -(\lambda x+y).
	$
	The origin is a R-singularity of second kind. Moreover, it is a hyperbolic focus 
	for the adjoint vector field whose eigenvalues are $\frac{-1\pm\sqrt{-3-4\lambda^{2}}}{2}$. 
	See figure \ref{fig-exemplo-singtipo2-foco}.\\

\noindent{\textbf{Example.}
	The polynomial $H(x,y,z) = y^{2}z - (x+y)$ is a first integral for
	 $X(x,y,z)=(\frac{1}{4},  -\frac{1}{2}(yz + \frac{1}{2}),z^2 )$
	and the flows on $M = H^{-1}(0)$ are described by
	$
	\dot{x} = \frac{1}{4}, \ y\dot{y} = -\frac{1}{2}(\frac{3y}{2} + x).$ The origin is a 
	R-singularity of first kind for the system whose eigenvalues are $-\frac{1}{2}$ and 
	$-\frac{1}{4}$. See figure \ref{fig-exemplo-singtipo1-no}.\\

\begin{figure}[h!]
	\center{\includegraphics[width=0.3\textwidth]{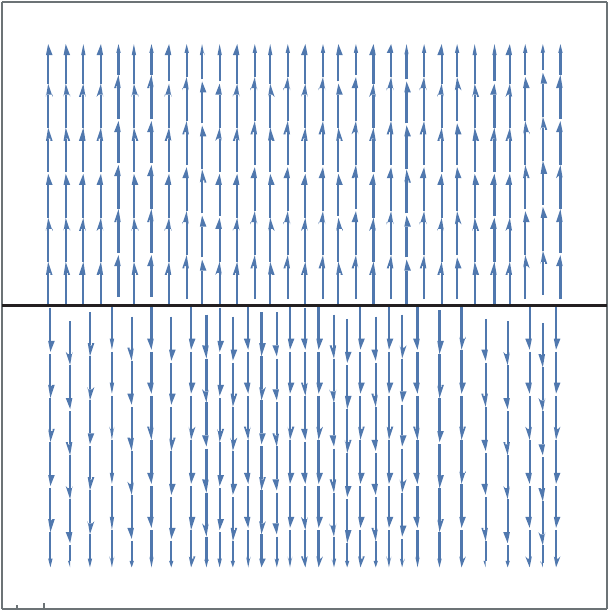}\hspace{0.5cm}\includegraphics[width=0.4\textwidth]{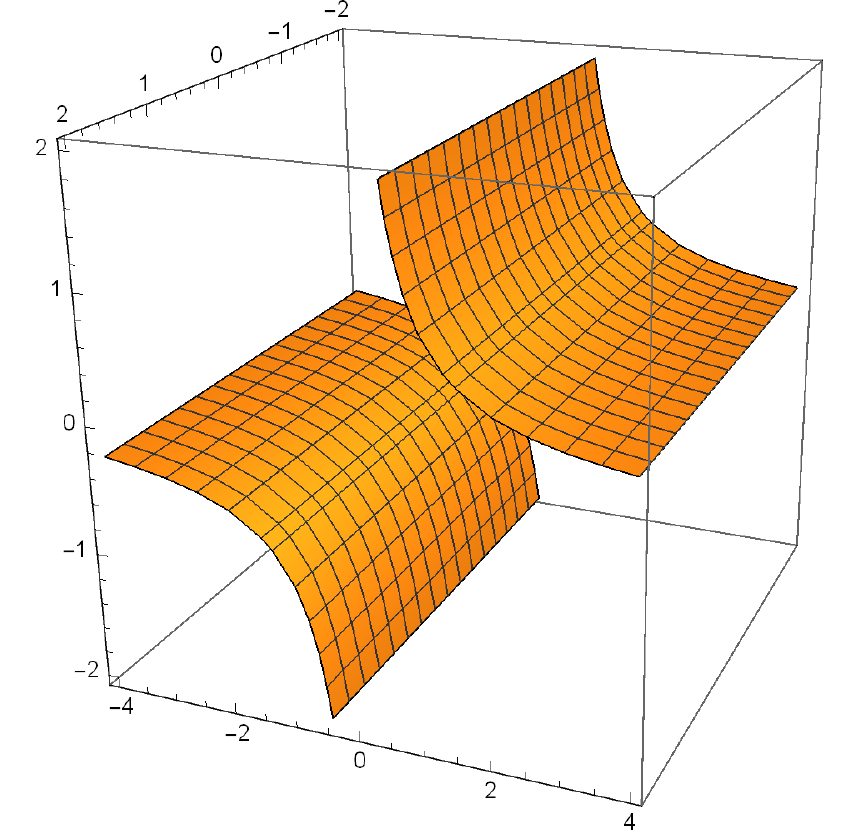}}
	\caption{\small{Phase portrait of $X(x,yz)=(0,z,-z^3)$ (left) and surface $M$ (right).}}\label{fig-exemplo-naosing}
\end{figure}

\begin{figure}[h!]
	\center{\includegraphics[width=0.3\textwidth]{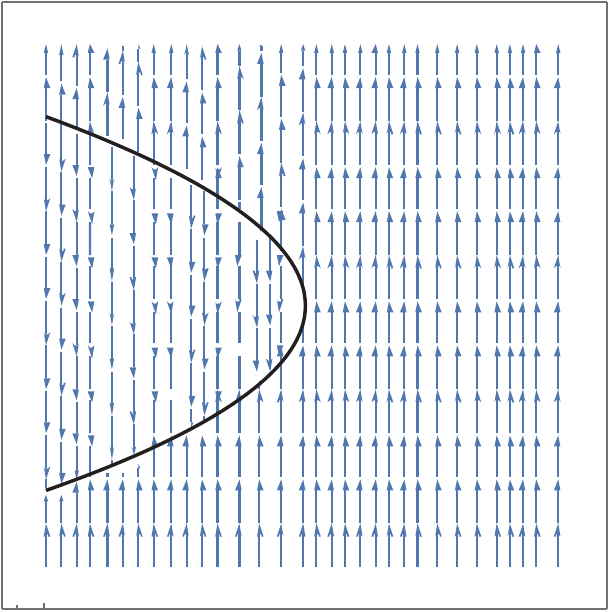}\hspace{0.5cm}\includegraphics[width=0.4\textwidth]{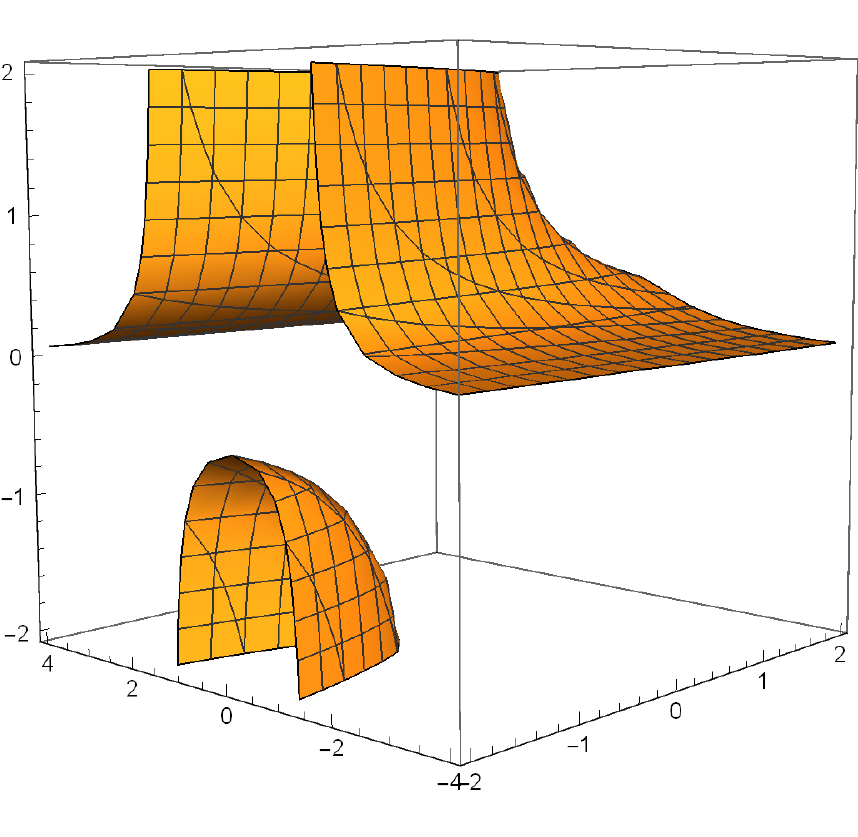}}
	\caption{\small{Phase portrait of $X(x,yz)=(0,z,-2yz^3)$ (left) and surface $M$ (right).}}\label{fig-exemplo-dobra}
\end{figure}

\begin{figure}[h!]
	\center{\includegraphics[width=0.25\textwidth]{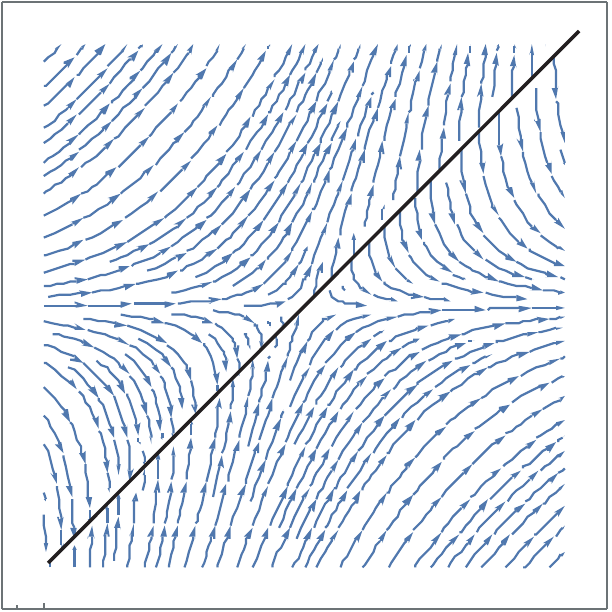}
		\includegraphics[width=0.25\textwidth]{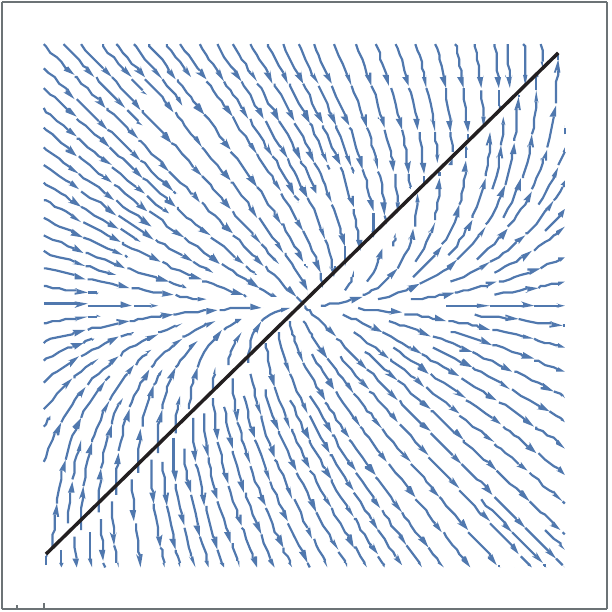}\includegraphics[width=0.3\textwidth]{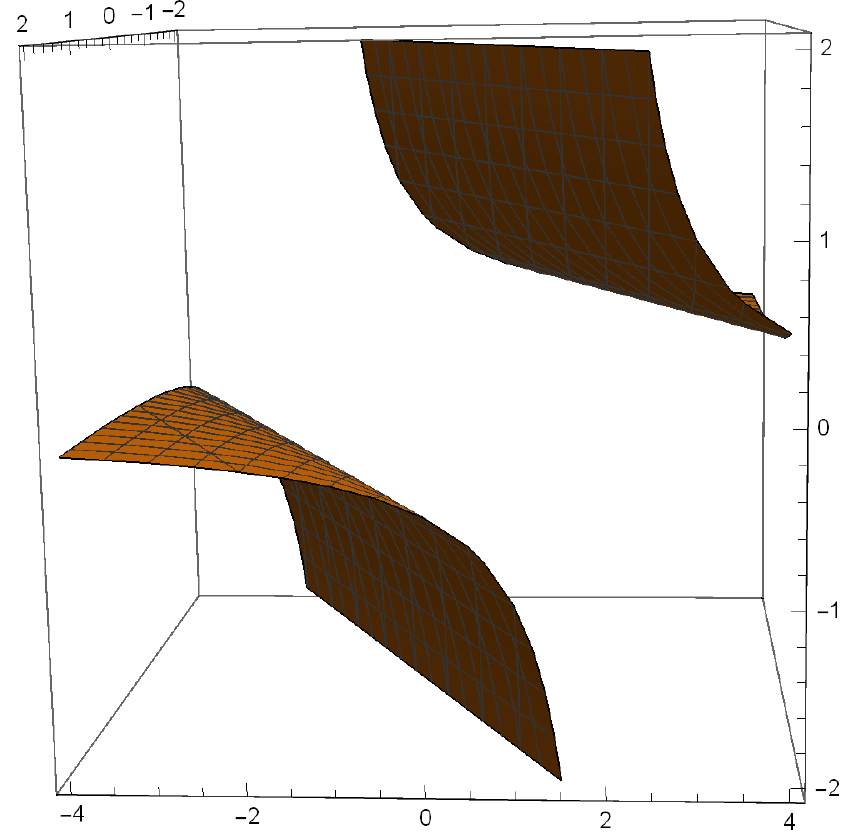}}
	\caption{\small{Phase portrait of $X(x,yz)=(1,\lambda yz,z^2-\lambda yz^3)$ for $\lambda > 1$, for $\lambda<1$ and surface $M$ (right).}}\label{fig-exemplo-singtipo2-sela}
\end{figure}

\begin{figure}[h!]
	\center{\includegraphics[width=0.3\textwidth]{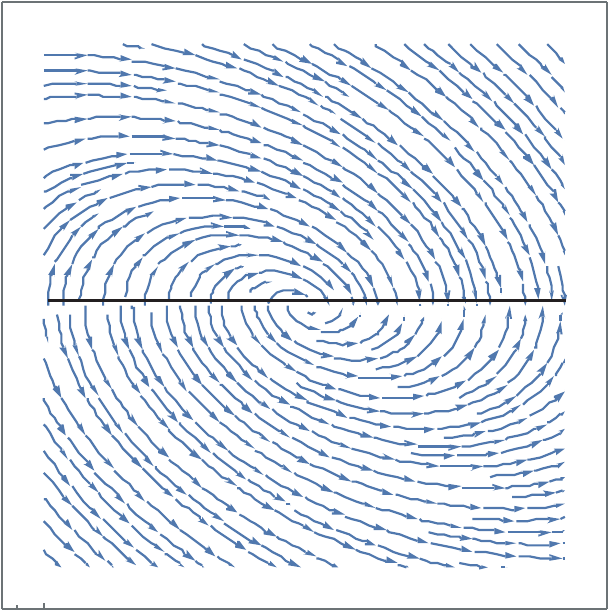}\hspace{0.5cm}\includegraphics[width=0.4\textwidth]{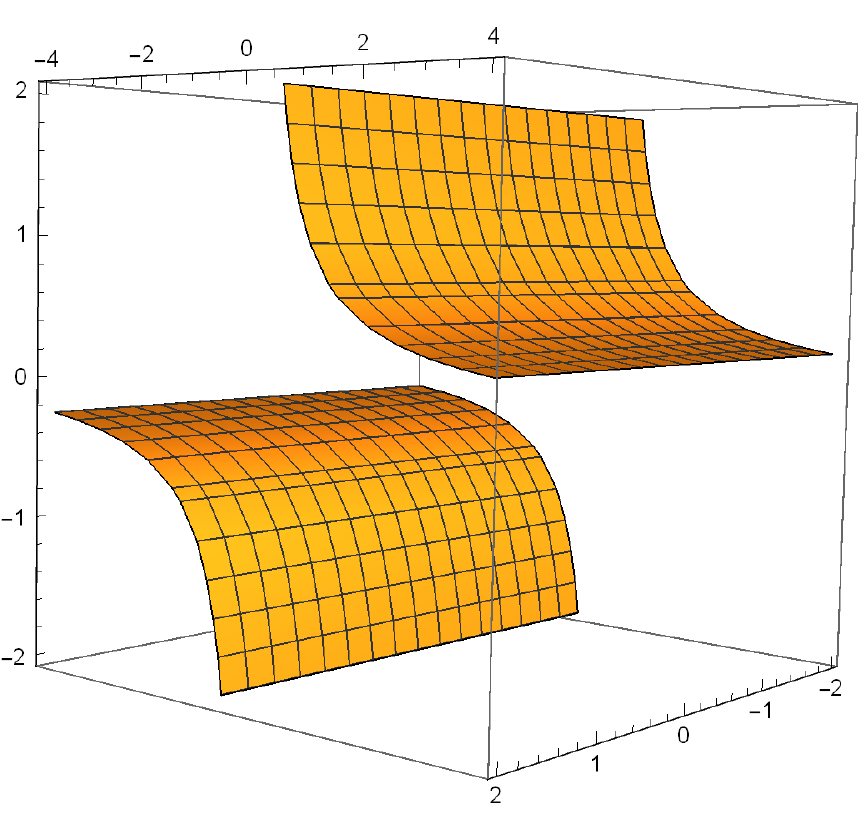}}
	\caption{\small{Phase portrait of $X(x,y,z)=(\frac{1 + \lambda^{2}}{\lambda},-(\lambda x+y)z, (\lambda x+y)z^{3})$ 
			(left) and surface $M$ (right).}}\label{fig-exemplo-singtipo2-foco}
\end{figure}

\begin{figure}[h!]
	\center{\includegraphics[width=0.3\textwidth]{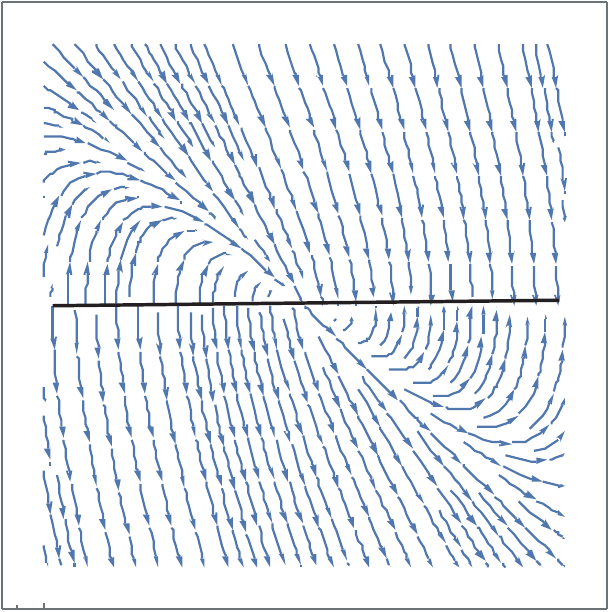}\hspace{0.5cm}\includegraphics[width=0.4\textwidth]{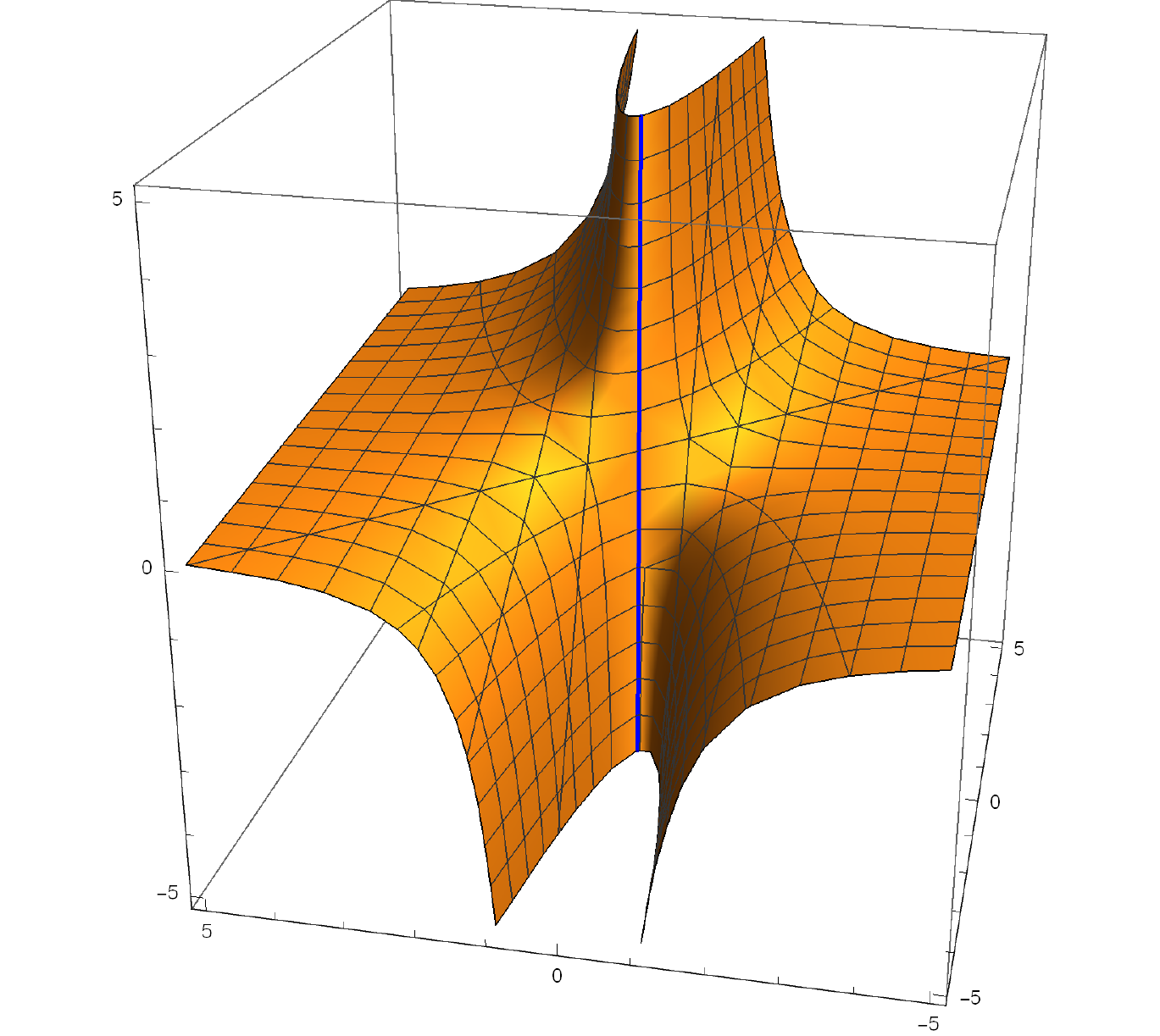}}
	\caption{\small{Phase portrait of $X(x,y,z)=(\frac{1}{4},  -\frac{1}{2}(yz + \frac{1}{2}),z^2 )$ (left) and surface $M$ (right). 
			The pseudo impasse set $\mathcal{I}_{M}$ is denoted in blue.}}\label{fig-exemplo-singtipo1-no}
\end{figure}

If $n \geq 2$ the linearization of the adjoint vector field \eqref{sis-impasse-casoparticular-regularizado} $J\widetilde{X}(x,y)$
 is given by
\begin{equation}\label{sis-impasse-matrixjacobiana-adjunto}
 \left(
\begin{array}{cc}
f^{n-1}(nf_{x}\alpha + f\alpha_{x}) & f^{n-1}(nf_{y}\alpha + f\alpha_{y}) \\
g^{n-1}(ng_{x}\beta_{n} + g\beta_{nx}) + \sum_{i = 0}^{n-1}S_{ix} & g^{n-1}(ng_{y}\beta_{n} + g\beta_{ny}) + \sum_{i = 0}^{n-1}S_{iy} \\
\end{array}
\right),
\end{equation}
where
$$S_{ix} = f^{n-i-1}g^{i-1}\Big{(} fg\beta_{ix} + \big{(} (n-i)f_{x}g + ifg_{x} \big{)}\beta_{i}\Big{)},$$ and
$$S_{iy} = f^{n-i-1}g^{i-1}\Big{(} fg\beta_{iy} + \big{(} (n-i)f_{y}g + ifg_{y} \big{)}\beta_{i}\Big{)}.$$

If $p\in\mathbb{R}^{2}$ is a R-singularity (of first or second kind),
then $p$ will be non-hyperbolic. In particular, if $p$ is R-singularity of first kind then $J\widetilde{X}(p)$ is identically zero.\\

\subsection{Particular case: $\beta(x,y)$ with degree 1 in the variable $z$.}   In what follows, we suppose that $X$ is
\begin{equation}\label{sis-impasse-grau1}
\dot{x} = \alpha(x,y), \ \dot{y} = \beta_{1}(x,y)z + \beta_{0}(x,y), \ \dot{z} = \gamma(x,y,z),
\end{equation}
whose flows are described by the constrained system
\begin{equation}\label{sis-impasse-casoparticular-grau1}
\dot{x} = \alpha(x,y), \ f(x,y)\dot{y} = g(x,y)\beta_{1}(x,y) + f(x,y)\beta_{0}(x,y),
\end{equation}
and its adjoint system  is
\begin{equation}\label{sis-impasse-casoparticular-regularizado-grau1}
\dot{x} = f(x,y)\alpha(x,y), \ \dot{y} = g(x,y)\beta_{1}(x,y) + f(x,y)\beta_{0}(x,y).
\end{equation}

The linearization  $J\widetilde{X}(x,y)$ of \eqref{sis-impasse-casoparticular-regularizado-grau1} is
\begin{equation}\label{sis-impasse-matrixjacobiana-adjunto1}
 \left(
\begin{array}{cc}
f_{x}\alpha + f\alpha_{x} & f_{y}\alpha + f\alpha_{y} \\
(f\beta_{0x} + g\beta_{1x}) + (f_{x}\beta_{0} + g_{x}\beta_{1}) & (f\beta_{0y} + g\beta_{1y}) + (f_{y}\beta_{0} + g_{y}\beta_{1}) \\
\end{array}
\right).
\end{equation}

Another justify for our choice is that all the examples in the next section (Falkner-Skan Equation,
 Lorenz System and Chen System) are in the form \eqref{sis-impasse-grau1}. Moreover, 
 linear vector fields can be written as \eqref{sis-impasse-grau1}.

It is important to observe that it may exist a function $\widetilde{f}:\mathbb{R}^{2}\rightarrow\mathbb{R}$ 
such that $f(x,y) = \widetilde{f}(x,y)\beta_{1}(x,y)$. In fact, this is the case for Lorenz System and 
Chen System. This implies that System \eqref{sis-impasse-grau1} gives rise to the constrained system
\begin{equation}\label{sis-impasse-casoparticular-grau1-obs}
\dot{x} = \alpha(x,y), \ \widetilde{f}(x,y)\dot{y} = g(x,y) + \widetilde{f}(x,y)\beta_{0}(x,y),
\end{equation}
whose adjoint vector field is
\begin{equation}\label{sis-impasse-casoparticular-regularizado-grau1-obs}
\dot{x} = \widetilde{f}(x,y)\alpha(x,y), \ \dot{y} = g(x,y) + \widetilde{f}(x,y)\beta_{0}(x,y).
\end{equation}

System  \eqref{sis-impasse-casoparticular-grau1-obs} does not have R-singularities of second
 kind and    $\mathcal{I}_{M}$ is
the union of two sets: lines that intersect the $xy$-plane at the points of $\{\widetilde{f}(x,y) = 0 = g(x,y)\}$ 
and the set of lines that intersect the $xy$-plane at the points of $\{\beta_{1}(x,y) = 0 = g(x,y)\}$. 
System \eqref{sis-impasse-casoparticular-grau1-obs}  describes the flow on the first set only.

In Lorenz and Chen Systems, the sets $\{\widetilde{f}(x,y) = 0\}$ and $\{\beta_{1}(x,y) = 0\}$ are the same. 
Therefore, we will focus in this case.

\begin{proposition}\label{prop-singularidades-impasses}
Consider system \eqref{sis-impasse-casoparticular-grau1} and let $r = \{(x_{0},y_{0},z)\}\subset\mathcal{I}_{M}$
 be a line. Then
\begin{enumerate}
\item The point $q = (x_{0},y_{0})$ is a RK-singularity of first kind or a R-singularity of first and second kind.
\item If $q$ is a RK-singularity of first kind, then $q$ is a hyperbolic equilibrium point if, and only if, $q$ is 
a node with one eigenvalue.
\item If $q$ is a R-singularity of first and second kind, then $q$ is non hyperbolic.
\end{enumerate}
\end{proposition}
\begin{proof}
Since $H(x,y,z) = f(x,y)z-g(x,y)$ defines an invariant surface for \eqref{sis-impasse-grau1}, then the equation
$z^{2}f_{y}\beta_{1} + z\Big{(}\alpha f_{x} + \beta_{0}f_{y} - \beta_{1}g_{y}\Big{)} - \Big{(}\alpha g_{x} + b_{0}g_{y}\Big{)} +\gamma f = \mu H$
holds for $(x,y,z)\in\mathbb{R}^{3}$. In particular, this still true for points in the pseudo-impasse set 
$\mathcal{I}_{M}$, that is, points such that $f(x,y) = g(x,y) = 0$. Then we have
$z^{2}f_{y}\beta_{1} + z\Big{(}\alpha f_{x} + \beta_{0}f_{y} - \beta_{1}g_{y}\Big{)} - \Big{(}\alpha g_{x} + \beta_{0}g_{y}\Big{)} = 0.$
In particular, we have a polynomial in the $z$ variable which is identically zero. 
Therefore its coefficients are identically zero for all $(x,y)\in\mathbb{R}^{2}$, and it implies that
\begin{equation}\label{eq-singularidades-impasse}
f_{y}\beta_{1} = 0, \ \alpha f_{x} + \beta_{0}f_{y} - \beta_{1}g_{y} = 0, \ \alpha g_{x} + \beta_{0}g_{y} = 0.
\end{equation}

The first equation in \eqref{eq-singularidades-impasse} tells us that a R-singularity of 
first kind $q = (x_{0}, y_{0})$ satisfies $f_{y}(q) = 0$ or $\beta_{1}(q) = 0$, and therefore 
$q$ is a RK-singularity of first kind or a R-singularity of first and second kind. 
For the second statement, if $q$ is RK-singularity of first kind then the linearization of 
\eqref{sis-impasse-casoparticular-regularizado-grau1} is
$
J\widetilde{X}(x,y) = \left(
\begin{array}{cc}
f_{x}\alpha & 0 \\
(f_{x}\beta_{0} + g_{x}\beta_{1}) & g_{y}\beta_{1} \\
\end{array}
\right),
$
and it implies that $q$ is hyperbolic if, and only if, $f_{x}\alpha$ and $g_{y}\beta_{1}$ are non-zero. 
Moreover, the second equation in \eqref{eq-singularidades-impasse} assures that 
$f_{x}\alpha = g_{y}\beta_{1}$, and then $q$ is a hyperbolic node for the adjoint 
vector field \eqref{sis-impasse-casoparticular-regularizado-grau1}. Finally, for the third statement,
 if $q$ is a R-singularity of first and second kind then the linearization of \eqref{sis-impasse-casoparticular-regularizado-grau1} is
$
J\widetilde{X}(x,y) = \left(
\begin{array}{cc}
f_{x}\alpha & f_{y}\alpha \\
f_{x}\beta_{0} & f_{y}\beta_{0} \\
\end{array}
\right),
$
and its eigenvalues are 0 and $f_{x}\alpha + f_{y}\beta_{0}$. Therefore $q$ is non hyperbolic.
\end{proof}

\begin{remark}
Proposition \ref{prop-singularidades-impasses} concerns systems of the form 
\eqref{sis-impasse-casoparticular-grau1}, and in general the statements 2 and 3 
are not true for systems of the form \eqref{sis-impasse-casoparticular-grau1-obs}. 
In fact, according our examples, systems like \eqref{sis-impasse-casoparticular-grau1-obs} 
can have a hyperbolic node of the adjoint system with two eigenvectors.
\end{remark}

\begin{theorem}\label{teo-fluxo-1}
Suppose that $H:\mathbb{R}^{3}\rightarrow\mathbb{R}$ given by $H(x,y,z) = f(x,y)z - g(x,y)$ 
defines an algebraic invariant surface for \eqref{sis-impasse-grau1}.
Assume that $\mathcal{Z}_{f}$ intercepts $\mathcal{Z}_{g}$ at isolated points.
\begin{enumerate}
\item Let $p = (x_{0}, y_{0}, z_{0})\in\mathcal{I}_{M}$ be a point and let $r\subset\mathcal{I}_{M}$ 
be the line through $p$. If $X(p)\in r$ then $q = (x_{0}, y_{0})$ is a non hyperbolic equilibrium 
point of the adjoint vector field \eqref{sis-impasse-casoparticular-regularizado-grau1}.
\item If $p = (x_{0},y_{0},z_{0})$ is a singular point for $M$ then $q = (x_{0},y_{0})$ is a non 
hyperbolic equilibrium point for the adjoint vector field \eqref{sis-impasse-casoparticular-regularizado-grau1}.
\item If $q = (x_{0}, y_{0})$ is a hyperbolic RK-singularity of first kind for the adjoint system 
\eqref{sis-impasse-casoparticular-regularizado-grau1}, then the line 
$r = \{(q, z), z\in\mathbb{R}\}\subset \mathcal{I}_{M}$ does not contain any equilibrium point of \eqref{sis-impasse-grau1}.
\end{enumerate}
\end{theorem}
\begin{proof}
Since $X(p)\in r$, then $X(p) = \big{(}0,0,\gamma(p)\big{)}$. In particular, 
we have $\alpha(x_{0}, y_{0}) = 0$ and therefore the linearization $J\widetilde{X}(x_{0},y_{0}) $  
of the adjoint system  \eqref{sis-impasse-casoparticular-regularizado-grau1} at $(x_{0}, y_{0})$ is
$ \left(
\begin{array}{cc}
0 & 0 \\
(f_{x}\beta_{0} + g_{x}\beta_{1}) & (f_{y}\beta_{0} + g_{y}\beta_{1}) \\
\end{array}
\right).
$
The eigenvalues are $0$ and $f_{y}\beta_{0} + g_{y}\beta_{1}$, thus $q$ is not hyperbolic. 
For the second statement, Corollary \ref{coro-superficie} assures that $q$ is R-singularity of 
first kind and therefore $q$ is an equilibrium point for the adjoint system 
\eqref{sis-impasse-casoparticular-regularizado-grau1}. Since $p$ is singular, we have the equations
$
f_{x}(x_{0},y_{0})z_{0} = g_{x}(x_{0},y_{0}), \ f_{y}(x_{0},y_{0})z_{0} = g_{y}(x_{0},y_{0}), \ f(x_{0},y_{0}) = 0,
$
and therefore the linearization $J\widetilde{X}(q)$ of \eqref{sis-impasse-casoparticular-regularizado-grau1} 
computed in $q$ is
$
\left(
\begin{array}{cc}
f_{x}\alpha & f_{y}\alpha \\
f_{x}(\beta_{0} + z_{0}\beta_{1}) & f_{y}(\beta_{0} + z_{0}\beta_{1}) \\
\end{array}
\right).
$

The eigenvalues are $0$ and $f_{x}\alpha + f_{y}(\beta_{0} + \beta_{1}z_{0})$. It follows that $q$ 
is non hyperbolic. Finally, for the third statement, since $q = (x_{0}, y_{0})$ is a hyperbolic 
R-singularity of first kind then $f(q) = g(q) = f_{y}(q) = 0$. The linearization $J\widetilde{X}(x,y)$ 
of \eqref{sis-impasse-casoparticular-regularizado-grau1} computed at $q$ is
$\left(
\begin{array}{cc}
f_{x}\alpha & 0 \\
(f_{x}\beta_{0} + g_{x}\beta_{1}) & g_{y}\beta_{1} \\
\end{array}
\right).
$

In particular, $\alpha(q) \neq 0$ and therefore the line $r = \{(x_{0},y_{0}, z), z\in\mathbb{R}\}\subset \mathcal{I}_{M}$ 
does not contain any equilibrium point of system \eqref{sis-impasse-grau1}.
\end{proof}

\begin{corollary}
	Let $r\subset\mathcal{I}_{M}$ be an invariant line of \eqref{sis-impasse-grau1}. 
	Then the projection of $r$ on the $xy$-plane is a non hyperbolic 
	equilibrium point for \eqref{sis-impasse-casoparticular-regularizado-grau1}.
\end{corollary}
\begin{proof}
	Note that for all $p\in r$,  $X(p)\in r$. By the first item of the Theorem \eqref{teo-fluxo-1}, the Corollary is true.
\end{proof}

\begin{corollary}\label{teo-fluxo-2-hiperbolico}
If $q = (x_{0}, y_{0})$ is a hyperbolic RK-singularity of first kind of the adjoint 
system \eqref{sis-impasse-casoparticular-regularizado-grau1}
then the line $r = \{(x_{0}, y_{0}, z)|z\in\mathbb{R}\}\subset \mathcal{I}_{M}$ satisfies:
	\begin{enumerate}
		\item The vector field \eqref{sis-impasse-grau1} is transversal to $r$.
		\item There are no singular point of $M$ in $r$.
		\item The line $r$ does not contain any equilibrium points of system \eqref{sis-impasse-grau1}.
	\end{enumerate}
\end{corollary}

The next results concern to the flow in a line $r\subset\mathcal{I}_{M}$ which its 
projection is a non hyperbolic point for the adjoint vector field \eqref{sis-impasse-casoparticular-regularizado-grau1}.

\begin{theorem}\label{teo-fluxo-nao-hip-2}
Let $q = (x_{0}, y_{0})\in\mathcal{I}$ be an isolated R-singularity of \eqref{sis-impasse-casoparticular-regularizado-grau1} 
and $r$ be the line through $q$ such that $r\subset\mathcal{I}_{M}$.
\begin{enumerate}
  \item If $\nabla f(q)$ and $\nabla g(q)$ are linearly independent and $q$ is of first and second kind, 
  then $r$ is invariant by the flows of \eqref{sis-impasse-grau1}. Moreover, the linearization of 
  \eqref{sis-impasse-casoparticular-regularizado-grau1} computed at $q$ is zero.
  \item If $\nabla f(q)$ and $\nabla g(q)$ are linearly independent and $q$ is a RK-singularity, 
  then $r$ is invariant by the flows of \eqref{sis-impasse-grau1} if, and only if, $\alpha(q)$ = 0.
  \item If $\nabla f(q)$ and $\nabla g(q)$ are linearly dependent, then there is a singular point of $M$ in $r$.
\end{enumerate}
\end{theorem}
\begin{proof}
Proposition \ref{prop-singularidades-impasses} implies that $q$ is a non hyperbolic point of 
the adjoint system \eqref{sis-impasse-casoparticular-regularizado-grau1}. Since $q$ is a
 R-singularity of first and second kind, then equation \eqref{eq-singularidades-impasse} becomes
$\alpha f_{x} + \beta_{0}f_{y} = 0, \ \alpha g_{x} + \beta_{0}g_{y} = 0.$ We can rewrite the previous equations as
${\langle}\big{(}\alpha(q), \beta_{0}(q)\big{)}, \nabla f(q){\rangle} = 0$, 
$ {\langle}\big{(}\alpha(q), \beta_{0}(q)\big{)}, \nabla g(q){\rangle} = 0.$
Since $\nabla f(q)$ and $\nabla g(q)$ are linearly independent, it follows  that 
$\big{(}\alpha(q), \beta_{0}(q)\big{)} = (0,0)$. Therefore we have 
$\alpha(q) = \beta_{0}(q) = \beta_{1}(q) = 0$ and $r$ is an invariant set for \eqref{sis-impasse-grau1}. 
Moreover, the linearization of \eqref{sis-impasse-casoparticular-regularizado-grau1} computed at $q$ is zero.
 For the second statement, since $q$ is a RK-singularity then $g_{y}(q)\neq 0$ and 
 equation \eqref{eq-singularidades-impasse} becomes
$\beta_{0} = -\frac{\alpha g_{x}}{g_{y}}, \ \beta_{1} = \frac{\alpha f_{x}}{g_{y}},$
and then $\alpha(q) = 0$ is a necessary and sufficient condition to assure that 
$\mathcal{I}_{M}$ is invariant. Finally, if $\nabla f(q)$ and $\nabla g(q)$ are 
linearly dependent then there is $c\neq0$ such that $\nabla g = c\nabla f$. 
Therefore we have
$\nabla H(q,z) = \Big{(}(z-c)f_{x}, (z-c)f_{y},0\Big{)},$
and then $(x_{0},y_{0},c)$ is a singular point of $M$.
\end{proof}

\begin{proposition}\label{teo-fluxo-nao-hip}
Let $p = (x_{0}, y_{0}, z_{0})\in\mathcal{I}_{M}$ be a singular point of $M$ and let 
$r$ be the line through $p$ such that $r\subset\mathcal{I}_{M}$. Then $X(p)\in r$ if, 
and only if, the linearization of the adjoint system
 \eqref{sis-impasse-casoparticular-regularizado-grau1} at $q = (x_{0},y_{0})$ is identically zero.
\end{proposition}
\begin{proof}
Since $p$ is singular, we have the equations
$
f_{x}(x_{0},y_{0})z_{0} = g_{x}(x_{0},y_{0})$, $ f_{y}(x_{0},y_{0})z_{0} = g_{y}(x_{0},y_{0})$,  $f(x_{0},y_{0}) = 0,
$
and therefore the linearization of \eqref{sis-impasse-casoparticular-regularizado-grau1} at $q$ is
$
J\widetilde{X}(q) = \left(
\begin{array}{cc}
f_{x}\alpha & f_{y}\alpha \\
f_{x}(\beta_{0} + z_{0}\beta_{1}) & f_{y}(\beta_{0} + z_{0}\beta_{1}) \\
\end{array}
\right).
$

Remember that we are supposing that $\mathcal{I}$ is regular, thus $f_{x}(q) \neq 0$ 
or $f_{y}(q)\neq 0$. Then $J\widetilde{X}(q)$ is identically null if, and only if, 
$\alpha(q) = 0$ and $\beta_{0}(q) + z_{0}\beta_{1}(q) = 0$.
\end{proof}

\begin{corollary}
Let $r = \{(x_{0},y_{0},z)|z\in\mathbb{R}\}\subset\mathcal{I}_{M}$ be a line 
parallel to the $z$-axis. If $r$ is a singular subset of $M$, then $r$ is invariant by 
\eqref{sis-impasse-grau1} if, and only if, the linearization of 
\eqref{sis-impasse-casoparticular-regularizado-grau1} at $q = (x_{0}, y_{0})$ is identically null.
\end{corollary}
\begin{proof}
Observe that for all $p\in r$ we have $X(p)\in r$.
\end{proof}

\section{Examples}\label{sec-examples}
In this section we apply our results to four well-known equations: Falkner-Skan equation (derived from fluid dynamics); 
Lorenz equation (meteorological studies); Chen equation (shows chaotic behavior) and  Fisher-Kolmogorov equation 
(related to population dynamics).

\subsection{Falkner-Skan equation}

\noindent

The Falkner-Skan equation was studied in \cite{FalknerSkan} and it is given by
$f''' + ff' + \lambda(1 - (f')^{2}) = 0,$
where $\lambda$ is a parameter. This equation describes a model in fluid dynamics and 
it describes a model of the steady two-dimensional flow of a slightly viscous incompressible 
fluid past a wedge. We can express this equation as a system of differential equations
\begin{equation}\label{sis-falkner}
\dot{x} = y, \ \dot{y} = z, \ \dot{z} = -xz -\lambda(1 - y^{2}).
\end{equation}

In \cite{LlibreValls} the authors proved that $H(x,y,z) = 2xz + (1 - y^{2})$ is the only 
Darboux polynomial of \eqref{sis-falkner} when $\lambda = \frac{1}{2}$. Observe that
${\langle} \Big{(}H_x,H_y,H_z\Big{)}, X {\rangle} = -xH(x,y,z).$
Denote $M = \{H(x,y,z) = 0\}$. The pseudo impasse set $\mathcal{I}_{M}\subset M$ is 
given by the lines $r_{1,2} = \{(0,\pm1,z),z\in\mathbb{R}\}$. Moreover, $M$ is a regular surface.

Let $f(x,y) = 2x$ and $g(x,y) = y^{2} -1$. The dynamics of \eqref{sis-falkner} on $M$ are described by
\begin{equation}\label{sis-falkner-plano}
\dot{x} = y, \ 2x\dot{y} = y^2 - 1,
\end{equation}
whose impasse curve is given by $\mathcal{I} = \{x = 0\}$. Its adjoint vector field is
\begin{equation}\label{sis-falkner-plano-adjunto}
\dot{x} = 2xy, \ \dot{y} = y^2 - 1,
\end{equation}
and the equilibrium points $p_{1,2} = (0, \pm 1)$ are RK-singularities of first kind on $\mathcal{I}$. 
Note that the points $p_{1,2}$ are the projections of the lines $r_{1,2}$. Observe that there are no 
equilibrium points of \eqref{sis-falkner-plano-adjunto} outside $\mathcal{I}$, therefore there are
 no equilibrium points of \eqref{sis-falkner} on $\mathcal{G}_{M}$ when $\lambda = \frac{1}{2}$.

The linearization of \eqref{sis-falkner-plano-adjunto} is given by
$
J\widetilde{X}(x,y) = \left(
\begin{array}{cc}
2y & 2x \\
0 & 2y \\
\end{array}
\right),
$
and thus $p_{1}$ is an unstable hyperbolic node and $p_{2}$ is a stable hyperbolic node for the 
adjoint vector field \eqref{sis-falkner-plano-adjunto}. It follows from Theorem \ref{teo-fluxo-2-hiperbolico} that 
the flows of \eqref{sis-falkner} with $\lambda = \frac{1}{2}$ is transversal to $\mathcal{I}_{M}$ and $\mathcal{I}_{M}$ 
does not contain any equilibrium point of \eqref{sis-falkner}. See figure \ref{fig-exemplo-falkner}.

Moreover, note that when we project the flows of \eqref{sis-falkner} with $\lambda = \frac{1}{2}$ 
on the $xz$-plane we obtain the systems
$
\dot{x} = \pm\sqrt{xz + 1}, \ \dot{z} = 0,
$
where $\dot{z} = 0$ means that the trajectories of \eqref{sis-falkner} on $M$ are contained 
in planes parallel to the $xy$-plane. Another way to verify this fact is observing that  at 
the points of $M$ the system \eqref{sis-falkner} takes form
$
\dot{x} = y, \ \dot{y} = z, \ \dot{z} = 0.
$

\begin{figure}[h!]
	\center{\includegraphics[width=0.35\textwidth]{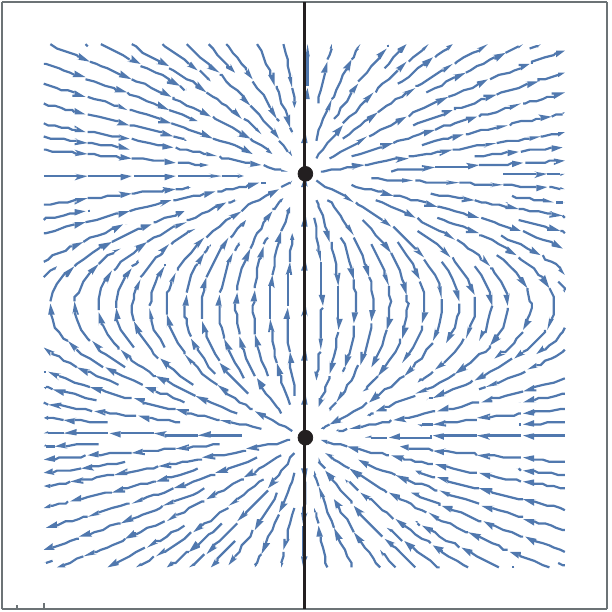}\hspace{0.5cm}\includegraphics[width=0.4\textwidth]{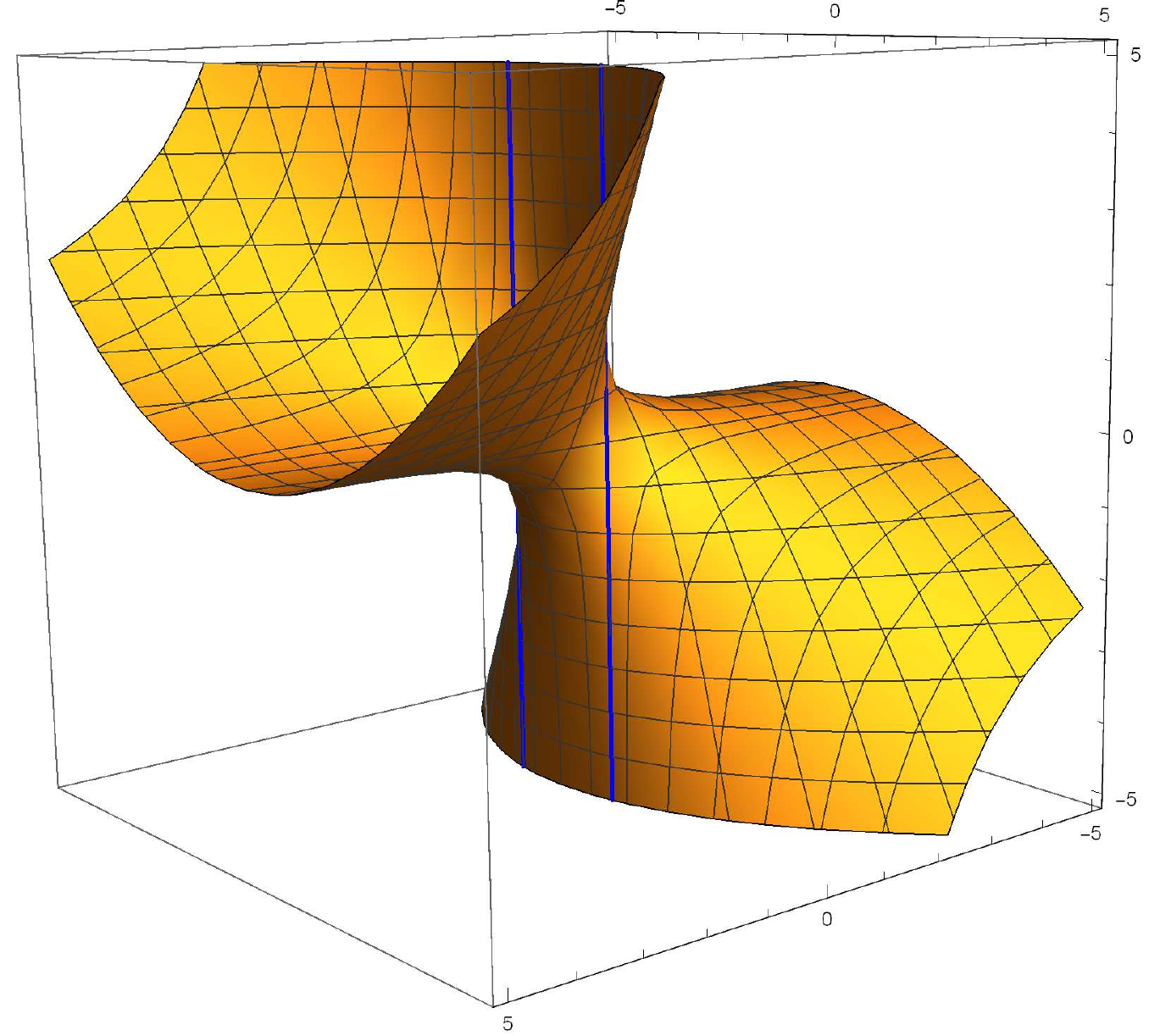}}
	\caption{\small{Phase portrait of \eqref{sis-falkner-plano} (left) and surface $M$ (right). The 
			pseudo impasse set $\mathcal{I}_{M}$ is denoted in blue.}}\label{fig-exemplo-falkner}
\end{figure}

\subsection{Lorenz System}

\noindent

The Lorenz System is given by
\begin{equation}\label{sis-Lorenz}
\dot{x} = s(-x+y), \ \dot{y} = rx - y - xz, \ \dot{z} = xy - bz,
\end{equation}
where $(x,y,z)\in\mathbb{R}^{3}$ are variables and $(s,r,b)\in\mathbb{R}^{3}$ are parameters. 
This model was proposed by Lorenz in 1963 (see \cite{Lorenz}) in order to study meteorological 
phenomena. When $(s,r,b) = (10,28, \frac{8}{3})$, this system presents chaotic behavior. 
In \cite{LlibreZhang} the authors gave all the six invariant algebraic surfaces for the 
Lorenz System \eqref{sis-Lorenz}. In \cite{CaoZhang} the flows on such invariant surfaces 
were considered and in \cite{LlibreMessiasSilva} the authors studied the global dynamics 
of \eqref{sis-Lorenz}.

Three out of six invariant algebraic surfaces can be written in the form 
$M = \{H^{-1}(0)\}$ with $H_{i}(x,y,z) = f_{i}(x,y)z-g_{i}(x,y)$, $i = 1, 2, 3$. They are:

\
\begin{center}
	\begin{tiny}
		\begin{tabular}{|c|c|c|c|}
			\hline
			Case & $(r,s,b)$ & $f_{i}(x,y)$ & $g_{i}(x,y)$ \\
			\hline
			(a) & $(r,0,\frac{1}{3})$ & $-\frac{4}{3}x^{2}$ & $\frac{4}{9}y^{2} + \frac{8}{9}xy - \frac{4}{3}rx^{2} - x^{4}$ \\
			(b) & $(r,1,4)$ & $-4x^{2} - 16(1-r)$ & $4rx^{2} - 8xy + 4y^{2} - x^{4}$ \\
			(c) & $(2s-1,s,6s-2)$ & $-4sx^{2}$ & $4s^{2}y^{2} - 4s(4s-2)xy + (4s-2)^{2}x^{2} - x^{4}$ \\
			\hline
		\end{tabular}
	\end{tiny}
\end{center}

\

In what follows, we analyze the flows of \eqref{sis-Lorenz} on such surfaces.

\textit{Case (a).} Since $(r,s,b) = (r,0,\frac{1}{3})$, system \eqref{sis-Lorenz} takes form
\begin{equation}\label{sis-lorenz-casoa-r3}
\dot{x} = \frac{1}{3}(-x+y), \ \dot{y} = rx - y - xz, \ \dot{z} = xy,
\end{equation}
and its flows on $M_{1} = \{H_{1}(x,y,z) = 0\}$ is described by the constrained system
\begin{equation}\label{sis-lorenz-casoa-impasse}
\dot{x}  =  \frac{1}{3}(-x+y), \ \frac{4}{3}x\dot{y}  =  \frac{4}{3}x(rx - y) + \frac{4}{9}y^{2} + \frac{8}{9}xy - \frac{4}{3}rx^{2} - x^{4},
\end{equation}
whose adjoint vector field is given by
\begin{equation}\label{sis-lorenz-casoa-adjunto}
\dot{x}  =  \frac{4}{9}x(-x+y), \ \dot{y}  =  \frac{4}{3}x(rx - y) + \frac{4}{9}y^{2} + \frac{8}{9}xy - \frac{4}{3}rx^{2} - x^{4}.
\end{equation}

The origin $(0,0)\in\mathcal{I}_{1} = \{x = 0\}$ is the only equilibrium point for \eqref{sis-lorenz-casoa-adjunto}. 
Since $\mathcal{I}_{M_{1}} = \{(0,0,z)| z\in\mathbb{R}\}$ is a singular subset of $M_{1}$ and 
the linearization of \eqref{sis-lorenz-casoa-adjunto} computed at $(0,0)$ is zero, 
it follows from Proposition \ref{teo-fluxo-nao-hip} that $\mathcal{I}_{M_{1}}$ is an 
invariant set of \eqref{sis-lorenz-casoa-r3}. This fact also can be verified observing 
that all points on $\mathcal{I}_{M_{1}}$ are equilibrium points. Notice that outside 
$\mathcal{I}_{1}$ there are no equilibrium points for \eqref{sis-lorenz-casoa-adjunto}, 
therefore there are no equilibrium points of \eqref{sis-lorenz-casoa-r3} on 
$\mathcal{G}_{M_{1}}$. See figure \ref{fig-exemplo-lorenz-casoa}.

\begin{figure}[h!]
	\center{\includegraphics[width=0.35\textwidth]{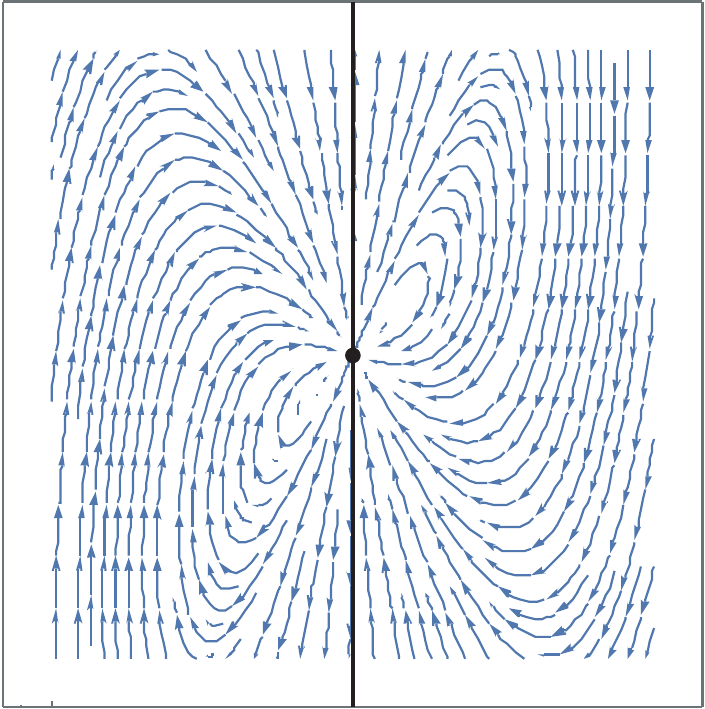}\hspace{0.5cm}\includegraphics[width=0.4\textwidth]{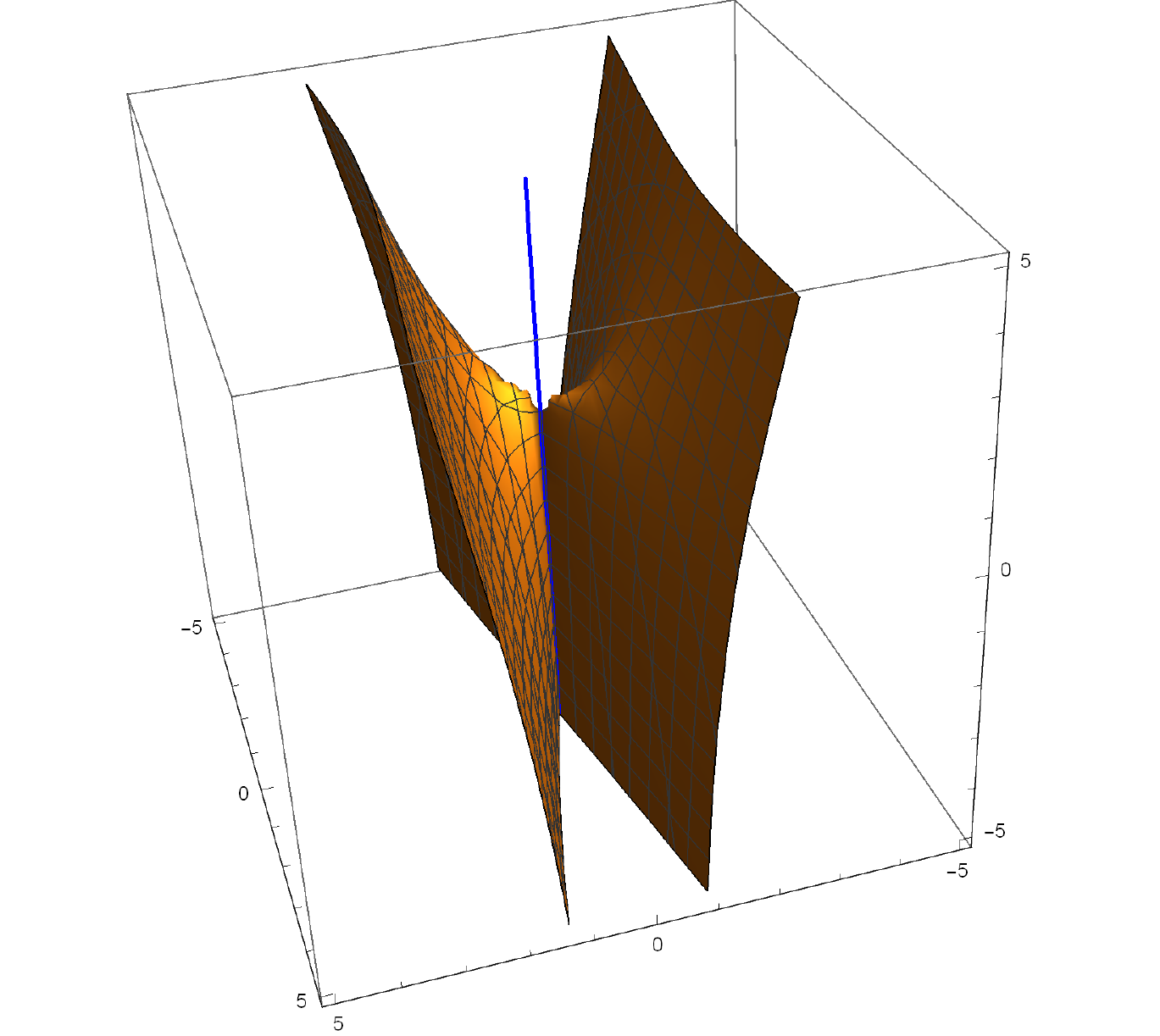}}
	\caption{\small{Phase portrait of \eqref{sis-lorenz-casoa-impasse} (left) and surface $M_{1}$ (right). 
			The pseudo impasse set $\mathcal{I}_{M}$ is denoted in blue.}}\label{fig-exemplo-lorenz-casoa}
\end{figure}

\textit{Case (b).} Suppose $(r,s,b) = (r,1,4)$. For $r < 1$ the function $f_{2}$ is always nonzero, for $r = 1$
 it has only one root and for $r > 1$ it has two roots. Since we are interested in the cases where $f_{2}$ 
 has real roots (because it is in this case that constrained systems rise), we will study the cases where $r = 1$ and $r > 1$.

For $r = 1$, system \eqref{sis-Lorenz} is
\begin{equation}\label{sis-lorenz-casob-r3-igual}
\dot{x} = -x+y, \ \dot{y} = x - y - xz, \ \dot{z} = -4z + xy,
\end{equation}
and its flows on $M_{2} = \{-4x^{2}z -(4x^{2} - 8xy + 4y^{2} - x^{4})= 0\}$ are described by
\begin{equation}\label{sis-lorenz-casob-impasse-igual}
\dot{x}  =  -x+y, \ 4x\dot{y} =  4x(x - y) + (4x^{2} - 8xy + 4y^{2} - x^{4}),
\end{equation}
whose adjoint system is
\begin{equation}\label{sis-lorenz-casob-adjunto-igual}
\dot{x}  =  4x(-x+y), \ \dot{y}  =  4x(x - y) + (4x^{2} - 8xy + 4y^{2} - x^{4}).
\end{equation}

The origin $(0,0)\in\mathcal{I}_{2} = \{x = 0\}$ is the only equilibrium point of \eqref{sis-lorenz-casob-adjunto-igual}. 
Since $\mathcal{I}_{M_{2}} = \{(0,0,z)| z\in\mathbb{R}\}$ is a singular set of $M_{2}$ and the 
linearization of \eqref{sis-lorenz-casob-adjunto-igual} computed at $(0,0)$ is zero, it follows 
from Proposition \ref{teo-fluxo-nao-hip} that $\mathcal{I}_{M_{2}}$ is invariant for \eqref{sis-lorenz-casob-r3-igual}. 
This fact also can be checked observing that at the points on $\mathcal{I}_{M_{2}}$ the vector field
 \eqref{sis-lorenz-casob-r3-igual} is $\big{(}0,0,-4z\big{)}$. See figure \ref{fig-exemplo-lorenz-casob-igual}.

For $r > 1$, system \eqref{sis-Lorenz} is
\begin{equation}\label{sis-lorenz-casob-r3-maior}
\dot{x} = -x+y, \ \dot{y} = rx - y - xz, \ \dot{z} = -4z + xy.
\end{equation}

Denote
$A(x,r) = -4x^{2} - 16(1-r), \ M_{2} = \big{\{}A(x,r)z -(4rx^{2} - 8xy + 4y^{2} - x^{4})= 0\big{\}}.$
The flow of \eqref{sis-lorenz-casob-r3-maior} on $M_{2}$ is described by the constrained system
\begin{equation}\label{sis-lorenz-casob-impasse-maior}
\dot{x} = -x+y, \ A(x,r)\dot{y} = A(x,r)(rx - y) - x(4rx^{2} - 8xy + 4y^{2} - x^{4}),
\end{equation}
whose adjoint system is
\begin{equation}\label{sis-lorenz-casob-adjunto-maior}
\dot{x} =  A(x,r)(-x+y), \ \dot{y} = A(x,r)(rx - y) - x(4rx^{2} - 8xy + 4y^{2} - x^{4}).
\end{equation}

In this case, $\mathcal{I}^{\pm}_{2} = \{(\pm2\sqrt{r-1}, y), y\in\mathbb{R}\}$ are impasse curves 
for system \eqref{sis-lorenz-casob-impasse-maior}. The pseudo impasse set in $M_{2}$ are the 
lines $$\mathcal{I}^{\pm}_{M_{2}} = \{(\pm2\sqrt{r-1},\pm2\sqrt{r-1},z), z\in\mathbb{R}\}.$$

Observe that $(\pm2\sqrt{r-1},\pm2\sqrt{r-1}, r-1)\in\mathcal{I}^{\pm}_{M_{2}}$ are singular points of 
$M_{2}$ and equilibrium points for \eqref{sis-lorenz-casob-r3-maior}. However, all the 
other points on $\mathcal{I}^{\pm}_{M_{2}}$ are regular points. System \eqref{sis-lorenz-casob-adjunto-maior} 
has three equilibrium points, and two of them are on the impasse curve. More precisely, 
$(\pm2\sqrt{r-1},\pm2\sqrt{r-1})\in\mathcal{I}^{\pm}_{2}$. The linearization of 
\eqref{sis-lorenz-casob-adjunto-maior} computed in such points is identically zero. 
The third equilibrium point of \eqref{sis-lorenz-casob-adjunto-maior} is the origin and 
it is a hyperbolic saddle. Therefore, there is a saddle point in $\mathcal{G}_{M_{2}}$. 
See figure \ref{fig-exemplo-lorenz-casob-maior}.

\begin{figure}[h!]
	\center{\includegraphics[width=0.35\textwidth]{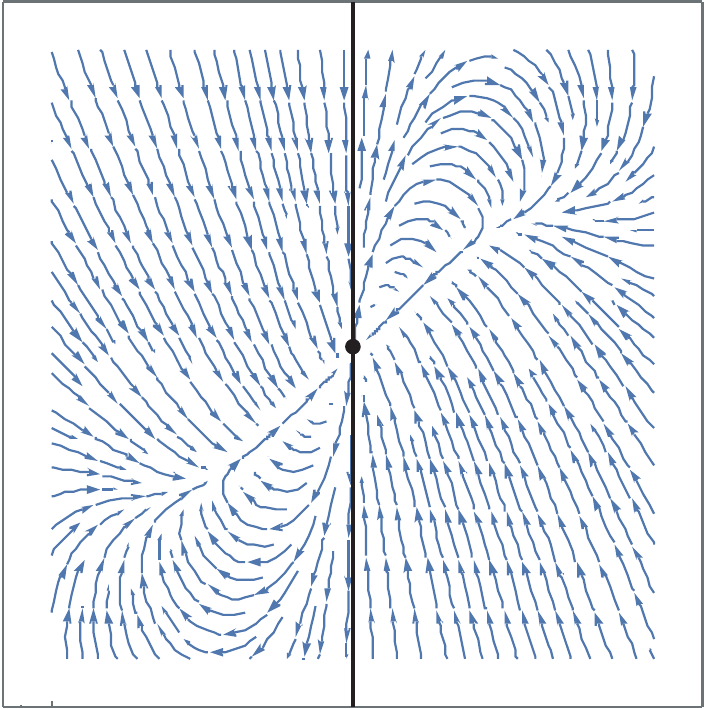}\hspace{0.5cm}\includegraphics[width=0.4\textwidth]{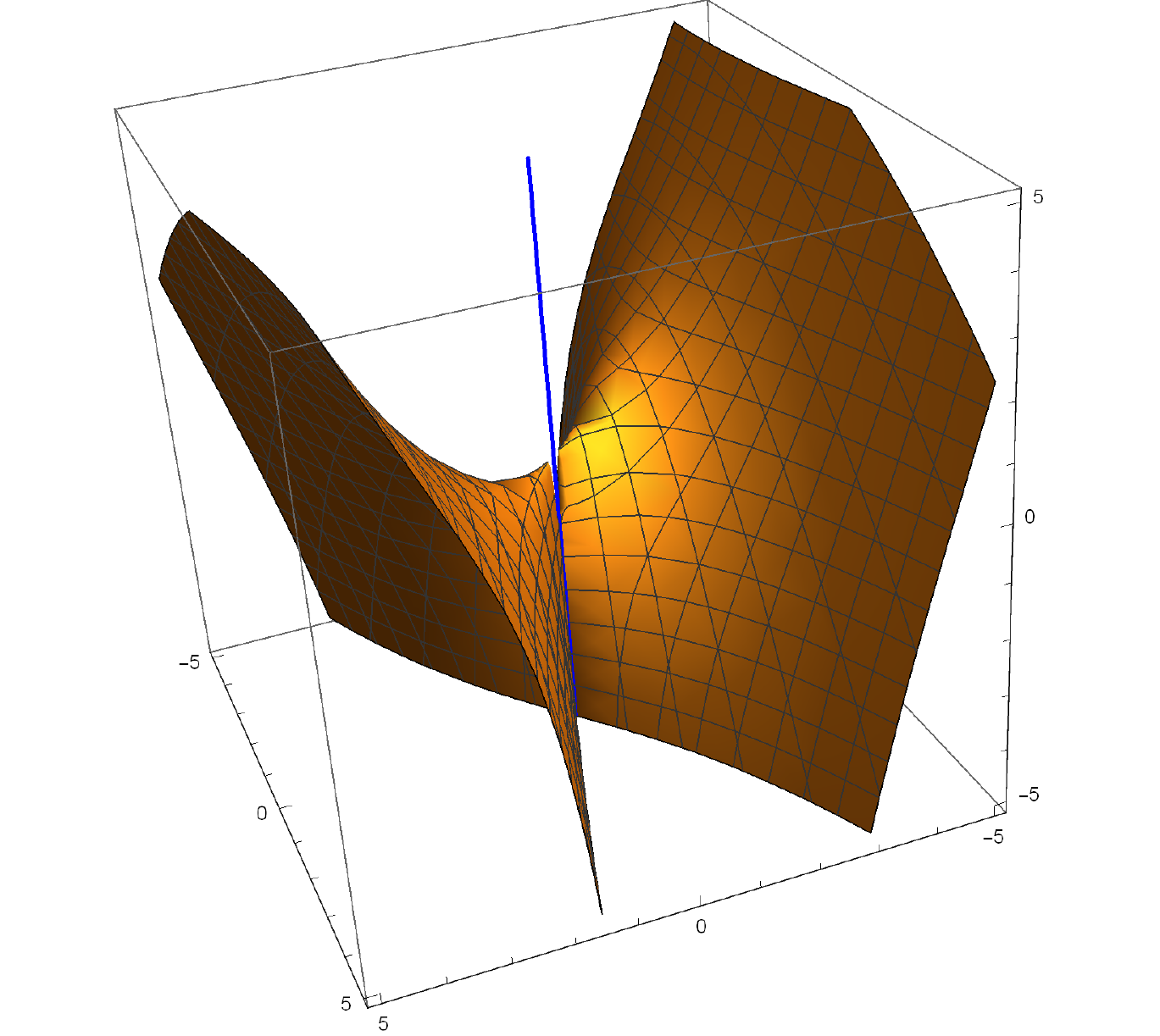}}
	\caption{\small{Phase portrait of \eqref{sis-lorenz-casob-impasse-igual} (left) and surface $M_{2}$ for $r = 1$ (right). 
			The pseudo impasse set $\mathcal{I}_{M}$ is denoted in blue.}}\label{fig-exemplo-lorenz-casob-igual}
\end{figure}

\begin{figure}[h!]
	\center{\includegraphics[width=0.35\textwidth]{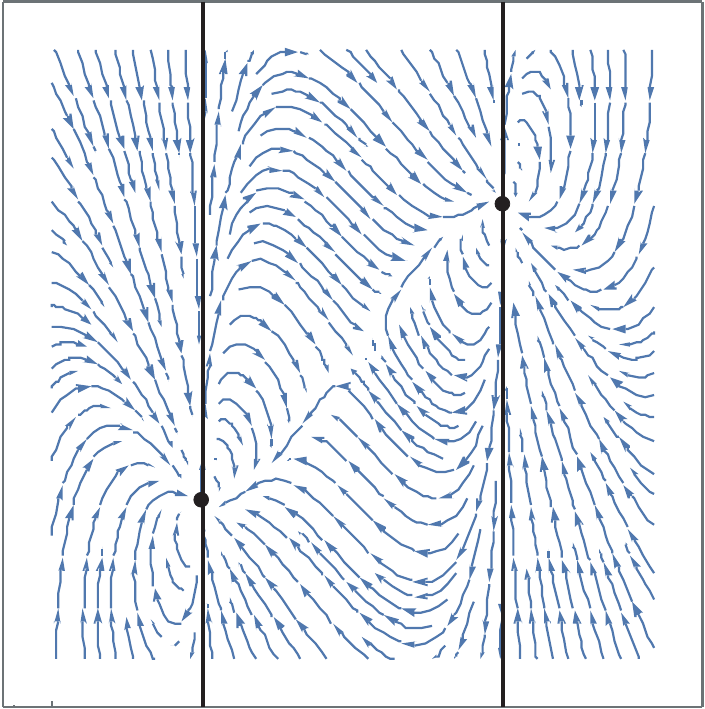}\hspace{0.5cm}\includegraphics[width=0.4\textwidth]{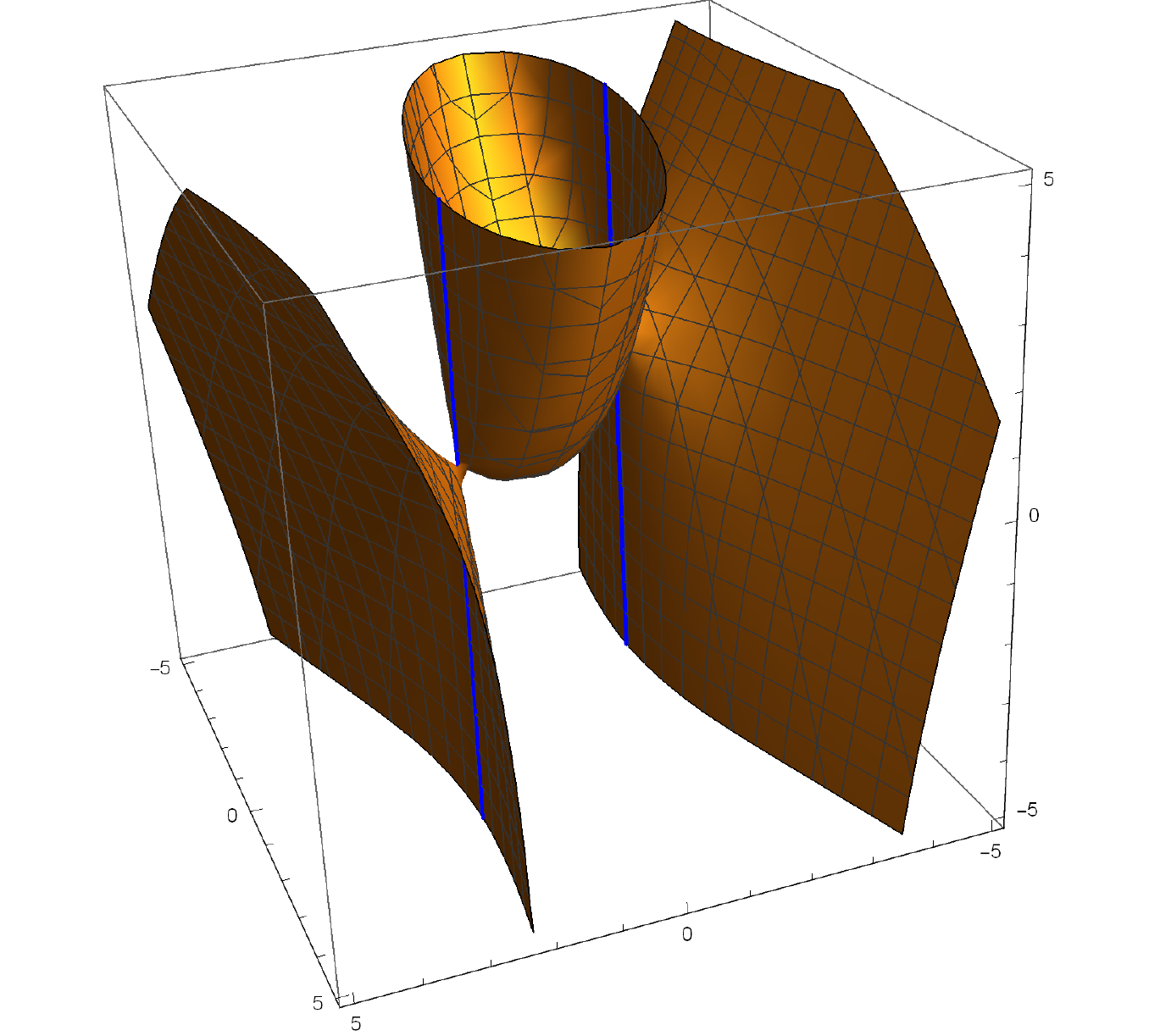}}
	\caption{\small{Phase portrait of system \eqref{sis-lorenz-casob-impasse-maior} (left) and 
			surface $M_{2}$ for $r > 1$ (right). The pseudo impasse set $\mathcal{I}_{M}$ is 
			denoted in blue.}}\label{fig-exemplo-lorenz-casob-maior}
\end{figure}

\textit{Case (c).} Since $(r,s,b) = (2s-1,s,6s-2)$, system \eqref{sis-Lorenz} takes form
\begin{equation}\label{sis-lorenz-casoc-r3}
\dot{x} = s(-x+y), \ \dot{y} = (2s-1)x - y - xz, \ \dot{z} = -(6s-2)z + xy,
\end{equation}

Denote
$g_{3}(x,y) = 4s^{2}y^{2} + (4s-2)^{2}x^{2} - 4s(4s-2)xy - x^{4}.$
The flow of \eqref{sis-lorenz-casoc-r3} on $M_{3} = \{H_{3}(x,y,z) = 0\}$ is given by
\begin{equation}\label{sis-lorenz-casoc-impasse}
\dot{x} = s(-x+y), \ 4sx\dot{y} = 4sx\big{(}(2s-1)x - y\big{)} + g_{3}(x,y),
\end{equation}
whose adjoint system is
\begin{equation}\label{sis-lorenz-casoc-adjunto}
\dot{x} = 4s^{2}x(-x+y), \ \dot{y} = 4sx\big{(}(2s-1)x - y\big{)} + g_{3}(x,y).
\end{equation}

The impasse curve is given by $\mathcal{I}_{3} = \{x = 0\}$ and the origin $(0,0)$ is the only equilibrium 
point for \eqref{sis-lorenz-casoc-adjunto}. Since $\mathcal{I}_{M_{3}} = \{(0,0,z)| z\in\mathbb{R}\}$ 
is a singular set of $M_{3}$ and the linearization of \eqref{sis-lorenz-casoc-adjunto} computed at $(0,0)$ 
is zero, it follows from Proposition \ref{teo-fluxo-nao-hip} that $\mathcal{I}_{M_{3}}$ is invariant by 
the flows of \eqref{sis-lorenz-casoc-r3}. This fact also can be checked observing that on 
$\mathcal{I}_{M_{3}}$ the system \eqref{sis-lorenz-casoc-r3} is $X = \big{(}0,0,-(6s-2)z\big{)}$.

Observe that outside $\mathcal{I}_{3}$ system \eqref{sis-lorenz-casoc-adjunto} has two equilibrium 
points given by $(\pm2\sqrt{1 + 3s^{2} - 4s}, \pm2\sqrt{1 + 3s^{2} - 4s})$. Such points are different
 when  $s > 1$ or $s < \frac{1}{3}$ and they collide at the origin when $s = 1$ or $s = \frac{1}{3}$. 
 Therefore there are two equilibrium points of \eqref{sis-lorenz-casoc-r3} on $\mathcal{G}_{M_{3}}$
  when $s > 1$ or $s < \frac{1}{3}$. See figure \ref{fig-exemplo-lorenz-casoc}.

\begin{figure}[h!]
	\center{\includegraphics[width=0.35\textwidth]{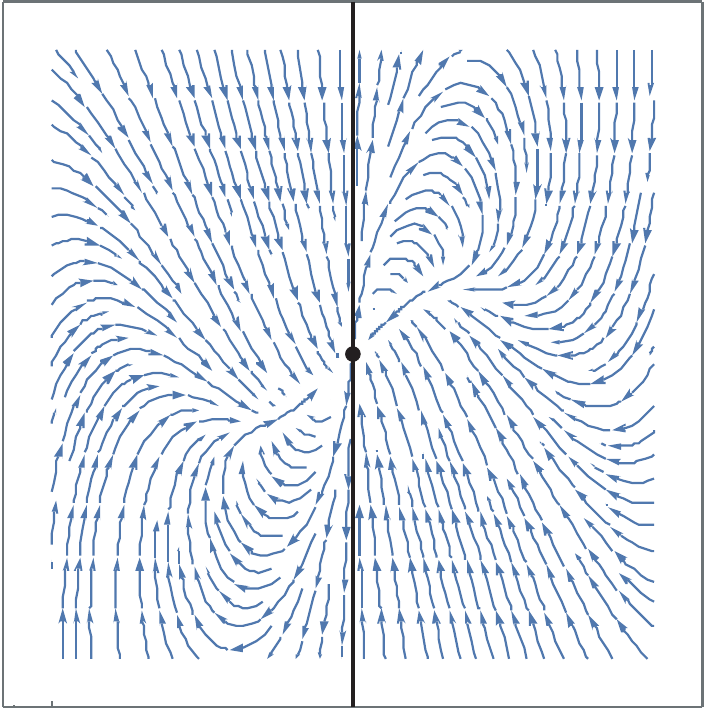}\hspace{0.5cm}\includegraphics[width=0.4\textwidth]{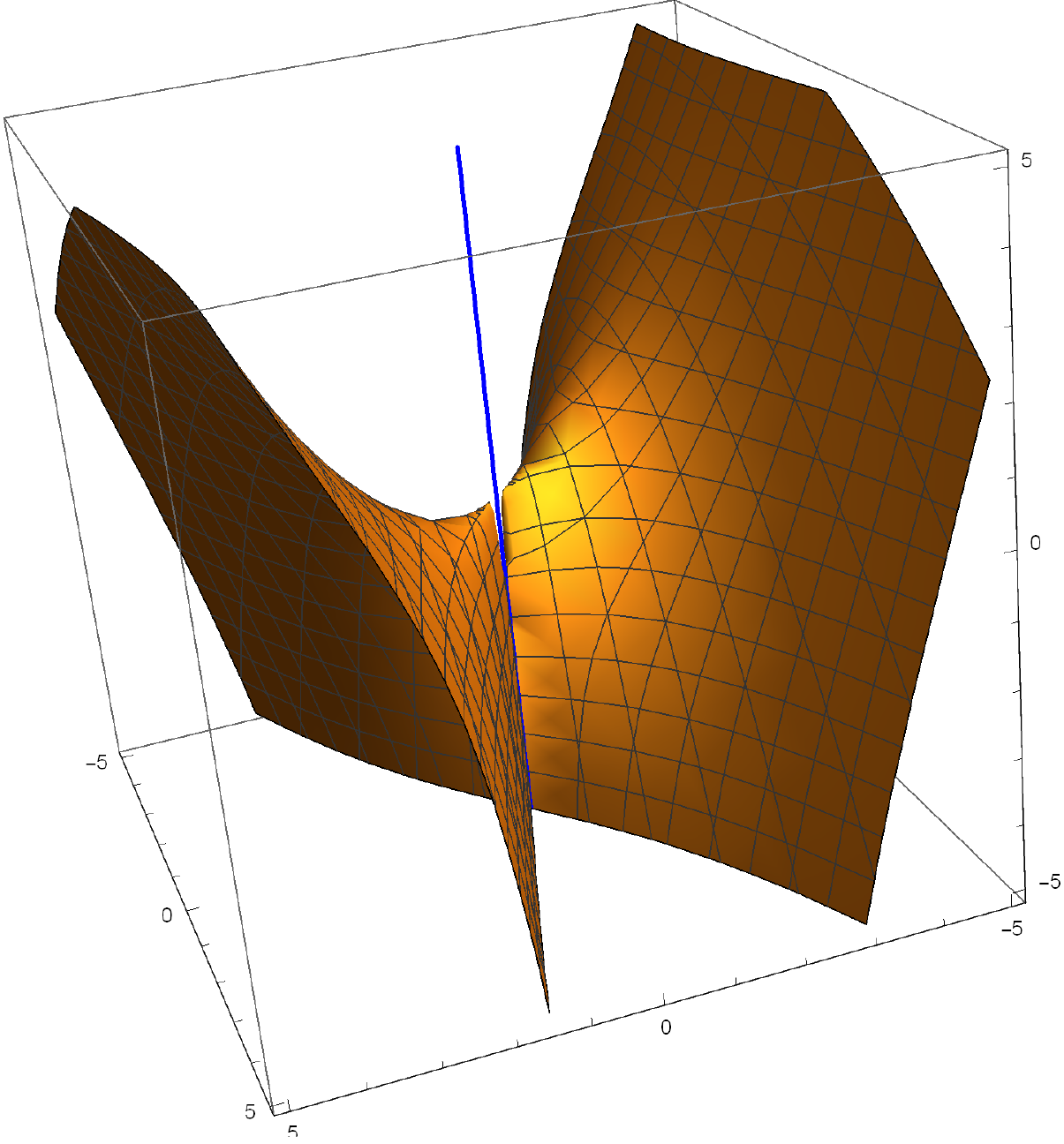}}
	\caption{\small{Phase portrait of System \eqref{sis-lorenz-casoc-impasse} (left) and surface $M_{3}$ for $s = 1$ (right). 
			The pseudo impasse set $\mathcal{I}_{M}$ is denoted in blue.}}\label{fig-exemplo-lorenz-casoc}
\end{figure}

\subsection{Chen System}

\noindent

Chen System is given by
\begin{equation}\label{sis-Chen}
\dot{x} = a(-x+y), \ \dot{y} = (c-a)x +cy - xz, \ \dot{z} = xy - bz,
\end{equation}
where $(x,y,z)\in\mathbb{R}^{3}$ are variables and $(a,b,c)\in\mathbb{R}^{3}$ are parameters. 
This model was proposed for the first time in \cite{ChenUeta} and it shows a chaotic behavior for a 
convenient choice of $a$, $b$ and $c$. In \cite{LuZhang} was given six invariant algebraic surfaces 
for Chen System \eqref{sis-Chen}. In \cite{CaoChenZhang} the authors studied the dynamics of 
\eqref{sis-Chen} on such surfaces and in \cite{LlibreMessiasSilva2012} the global dynamics of 
\eqref{sis-Chen} was considered.

Two out of six invariant algebraic surfaces studied in such references can be written in the form 
$M = \{H^{-1}_{i}(0)\}$ com $H_{i}(x,y,z) = f_{i}(x,y)z-g_{i}(x,y)$, $i = 4,5$. They are:

\
\begin{center}
	\begin{tiny}
		\begin{tabular}{|c|c|c|c|}
			\hline
			Case & $(a,b,c)$ & $f_{i}(x,y)$ & $g_{i}(x,y)$ \\
			\hline
			(d) & $(-\frac{c}{3}, 0, c)$ & $\frac{4}{3}cx^{2}$ & $\frac{16}{9}c^{2}x^{2} + \frac{8}{9}c^{2}xy + \frac{4}{9}c^{2}y^{2} - x^{4}$ \\
			(e) & $(-c,-4c,c)$ & $4cx^{2} + 48c^{3}$ & $4c^{2}y^{2} - 8c^{2}xy - 8c^{2}x^{2} - x^{4}$ \\
			\hline
		\end{tabular}
	\end{tiny}
\end{center}

\

Observe that for all $c\in\mathbb{R}$ the function $f_{5}(x,y) = 4cx^{2} + 48c^{3}$ is always non zero. 
This implies that the pseudo impasse set $\mathcal{I}_{M_{5}}$ is empty and the flows on $M_{5}$ 
are described by a smooth system.

When $(a,b,c) = (-\frac{c}{3}, 0, c)$, system \eqref{sis-Chen} is
\begin{equation}\label{sis-chen-caso1}
\dot{x} = -\frac{c}{3}(-x+y), \ \dot{y} = c(\frac{4}{3}x + y) - xz, \ \dot{z} = xy,
\end{equation}
and its flows on $M_{4} = \{H_{4}(x,y,z) = 0\}$ are described by
\begin{equation}\label{sis-chen-caso1-impasse}
\dot{x} = -\frac{c}{3}(-x+y), \ \frac{4}{3}cx\dot{y} = \frac{4}{3}c^{2}x\big{(}\frac{4}{3}x + y\big{)} - g_{5}(x,y),
\end{equation}
whose adjoint system is
\begin{equation}\label{sis-chen-caso1-adjunto}
\dot{x} = \frac{4}{9}c^{2}x(x-y), \ \dot{y} = \frac{4}{3}c^{2}x\big{(}\frac{4}{3}x + y\big{)} - g_{5}(x,y).
\end{equation}

The impasse curve is given by $\mathcal{I}_{4} = \{x = 0\}$ and the origin $(0,0)$ is the only equilibrium 
point for \eqref{sis-chen-caso1-adjunto}. Since $\mathcal{I}_{M_{4}} = \{(0,0,z)| z\in\mathbb{R}\}$ 
is a singular set of $M_{4}$ and the linearization of \eqref{sis-chen-caso1-adjunto} computed at $(0,0)$ 
is zero, it follows from Proposition \ref{teo-fluxo-nao-hip} that $\mathcal{I}_{M_{4}}$ is an invariant set for the 
flows of \eqref{sis-chen-caso1}. Observe that outside $\mathcal{I}_{4}$ there are no 
equilibrium points for \eqref{sis-chen-caso1-adjunto}, therefore there are no equilibrium 
points for \eqref{sis-chen-caso1} on $\mathcal{G}_{M_{4}}$. See figure \ref{fig-exemplo-chen}.

\begin{figure}[h!]
	\center{\includegraphics[width=0.35\textwidth]{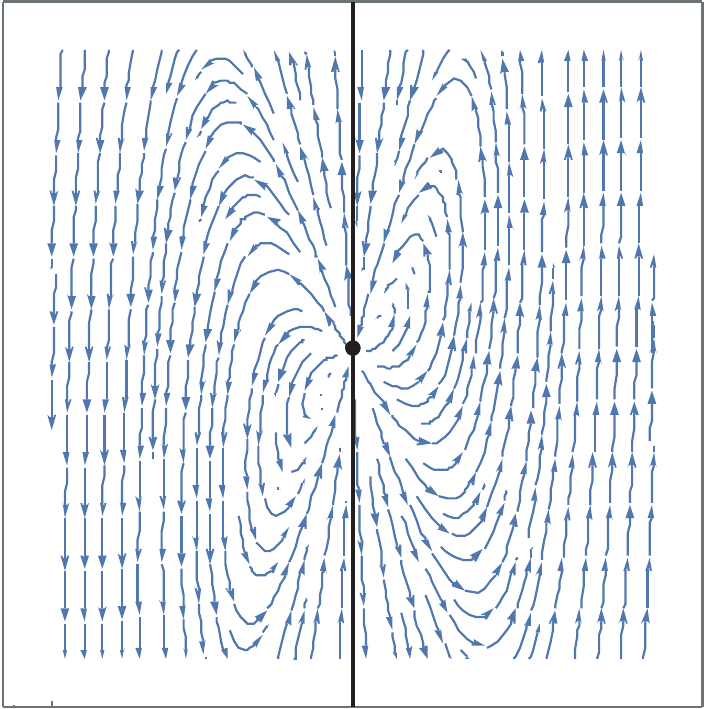}\hspace{0.5cm}\includegraphics[width=0.4\textwidth]{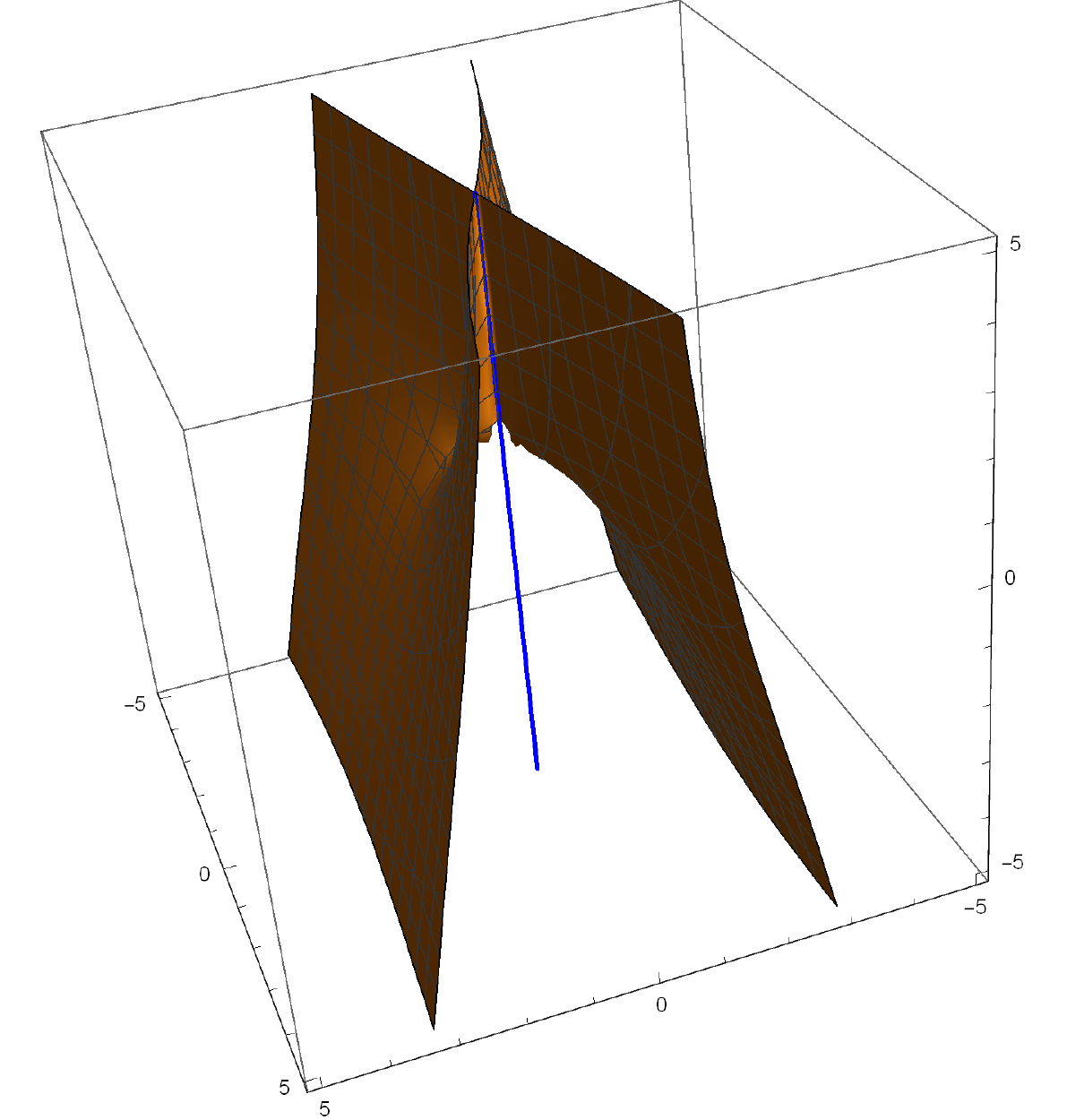}}
	\caption{\small{Phase portrait of system \eqref{sis-chen-caso1-impasse} (left) and surface $M_{4}$ (right). 
			The pseudo impasse set $\mathcal{I}_{M}$ is denoted in blue.}}\label{fig-exemplo-chen}
\end{figure}

\subsection{Constrained systems in higher dimensions and flows on invariant hypersurfaces}

In this section we discuss how the previous problems can be extended to higher dimensions. 
In order to do this, we illustrate our discussion by means of an example.

The Fisher-Kolmogorov Equation was introduced in \cite{Fischer} and it is a model for populational dynamics. 
Such equation is given by
$
u_t = u_{xx} + u - u^{3}.
$
In \cite{DeeSaarloos} the authors proposed the Extended Fisher-Kolmogorov Equation (EFK-equation for short) given by
$u_t=-\gamma u_{xxxx}+u_{xx} + u - u^{3}, \ \gamma > 0.
$

Observe that for $\gamma = 0$, EFK-equation becomes the regular Fisher-Kolmogorov equation. 
For stationary solutions (solutions in which do not depend on time $t$), the EFK-equation reduces to
$$
-\gamma u_{xxxx}+u_{xx} + u - u^{3} = 0, \ \gamma >0.
$$

Applying some transformations and changes of variables, we can express the stationary solutions of 
the EFK-equation by means of the polynomial system
\begin{equation}\label{EFK-system}
\dot{x} = y, \ \dot{y} = z, \ \dot{z} = w, \ \dot{w} = x - qz - x^{3},
\end{equation}
where $(x,y,z,w)\in \mathbb{R}^{4}$ and $q\in\mathbb{R}$ is negative.

In \cite{LlibreMessiasSilva2011} the authors proved that
$
H(x,y,z,w) = \frac{qy^{2} - x^{2} - z^{2}}{2} + \frac{x^{4}}{4} + wy
$
is a first integral for \eqref{EFK-system} and therefore $H^{-1}(k)$ defines an invariant algebraic 
hypersurface. For $k\neq 0$ and $k\neq -\frac{1}{4}$, $H^{-1}(k)$ is a smooth 3-dimensional hypersurface. 
Note that $H$ is written in the form $H(x,y,z,w) = f(x,y,z)w - g(x,y,z)$, and therefore we can 
define the sets $\mathcal{G}_{M}$ and $\mathcal{I}_{M}$, where $M = H^{-1}(k)$. Moreover, 
Proposition \ref{prop-superficie} is still true for $n$-dimensional hypersurfaces.

The authors also proved in \cite{LlibreMessiasSilva2011} that the flows on $H^{-1}(k)$ can be 
described by the constrained system
\begin{equation}\label{EFK-system-impasse}
\dot{x} = y, \ \dot{y} = z, \ 4y\dot{z} = 4k + 2(x^{2} + z^{2} - qy^{2}) - x^{4},
\end{equation}
whose adjoint vector field is
\begin{equation}\label{EFK-system-adjoint}
\dot{x} = 4y^{2}, \ \dot{y} = 4yz, \ \dot{z} = 4k + 2(x^{2} + z^{2} - qy^{2}) - x^{4}.
\end{equation}

The ideas used to prove this fact are the same used in the proof of Theorem \ref{teo-fluxo-impasse}. 
Note that the impasse surface of \eqref{EFK-system-impasse} is the $xz$-plane on $\mathbb{R}^{3}$, 
and the projection of the pseudo impasse set 
$\mathcal{I}_{M} = \Big{\{}(x,y,z,w) | y = 0 = \frac{qy^{2} - x^{2} - z^{2}}{2} + \frac{x^{4}}{4} + k, w\in\mathbb{R}\Big{\}}$ 
is a one-dimensional curve of equilibrium points for the adjoint vector field \eqref{EFK-system-adjoint}.

Although Proposition \ref{prop-superficie} and Theorem \ref{teo-fluxo-impasse} can 
be easily extended in higher dimensions, it is harder to generalize Theorem \ref{teo-fluxo-2-hiperbolico} 
and Proposition \ref{prop-singularidades-impasses}. In the 2-dimensional case, the projection of the 
pseudo impasse set is a set of equilibrium points contained in the impasse curve. In higher dimensions,
 if $M$ is a $n$-dimensional hypersurface embedded in $\mathbb{R}^{n+1}$, then $\mathcal{I}_{M}$ 
 and $\mathcal{I}$ will be a $(n-1)$-dimensional submanifold of $M$ and $\mathbb{R}^{n}$, respectively. 
 Moreover, the projection of $\mathcal{I}_{M}$ is a $(n-2)$-dimensional submanifold of $\mathcal{I}$ and 
 all their points are equilibrium points for the $n$-dimensional constrained system. This case require a detailed analysis.

\section{Unfolding minimal sets in 1-parameter families of invariant algebraic surfaces}\label{sec-spp}

In this section we consider 1--parameter families of smooth vector fields
\begin{equation}\label{1pf}
X_{\e}(x,\mathbf{y})=\Big{(}\alpha_{\e}(x,\mathbf{y}), \beta_{\e} (x,\mathbf{y})\Big{)}, \quad \e \downarrow 0,
\end{equation}
where $x\in\mathbb{R}$ and $\mathbf{y}\in\mathbb{R}^{m}$. Assume that $H_{\e}(x,\mathbf{y})$ 
is a smooth first integral and denote
\begin{equation}\label{ais}
M_{\e}=H_{\e}^{-1}(0),
\end{equation}
the correponding invariant surface. We also assume that the vector field and the first integral vary 
smoothly with respect to the parameter $\e$.

The trajectories of \eqref{1pf} are the solutions of
\begin{equation}\label{sys-family-general}
\dot{x} = \alpha_{\e}(x,\mathbf{y}), \ \dot{\mathbf{y}} = \beta_{\e}(x,\mathbf{y}).
\end{equation}

Using singular perturbation theory and Fenichel Theorem \ref{teo-Fenichel} as main tools, we  study 
the persistence of equilibrium points and periodic orbits using normal hyperbolicity.

Since $H_{\e}$ is a first integral of \eqref{sys-family-general}, for each $\e$
sufficiently small we have
$\alpha_{\e}(H_{\e})_x+
 \beta_{\e}(H_{\e})_{\mathbf{y}}  = 0.$
Thus we can rewrite system \eqref{sys-family-general} as
\begin{equation}\label{sys-first-integral}
\dot{x} = -\Bigg{(}\frac{\beta_{\e}(H_{\e})_{\mathbf{y}}}{(H_{\e})_x}\Bigg{)}, \ \dot{\mathbf{y}} = \beta_{\e}(x,\mathbf{y}).
\end{equation}

The smooth dependence on $\e$ implies that $\alpha_{\e}\rightarrow \alpha_{0}$, $\beta_{\e}\rightarrow \beta_{0}$
 and $H_{\e}\rightarrow H_{0}$ in the $C^1$-topology.

If $\varphi_{0}(t) = \big{(} x_{0}(t), \mathbf{y}_{0}(t)\big{)}$ is a solution of \eqref{sys-family-general} 
satisfying that $H_0(\varphi_{0}(0))=0$, then $\varphi_{0}(t)$ is a solution for the 
differential-algebraic equation
\begin{equation}\label{diff-algebraic-eq}
0 = H_{0}(x,\mathbf{y}), \ \dot{\mathbf{y}} = \beta_{0}(x,\mathbf{y}).
\end{equation}

Conversely, consider the differential-algebraic equation \eqref{diff-algebraic-eq} and let 
$p_{0}\in\mathcal{NH}(M_{0})$. There is a neighborhood $V\subset \mathcal{NH}(M_{0})$ of 
$p_{0}$ such that if $\varphi_{0}(t) = \big{(}x_{0}(t),\mathbf{y}_{0}(t)\big{)}$,  $\varphi_{0}(0) = p_{0}$
 is a solution of \eqref{diff-algebraic-eq} in $V$, then $\varphi_{0}(t)$ is a solution of
\begin{equation}\label{sys-first-integral0}
\dot{x} = -\Bigg{(}\frac{\beta_{0}(H_0)_\mathbf{y}}{(H_0)_x}\Bigg{)},
\ \dot{\mathbf{y}} = \beta_0(x,\mathbf{y}).
\end{equation}

The neighborhood $V$ is the neighborhood in which $(H_0)_x\neq0$. Differentiating 
$H_0(x, \mathbf{y}) = 0$ with respect to $t$, we get
$0 = \dot{x}(H_0)_x + \dot{\mathbf{y}}(H_0)_\mathbf{y}, \ \dot{\mathbf{y}} = \beta_{0}(x,\mathbf{y}),
$
as desired. We  summarize these previous facts in the following Proposition.

\begin{proposition}\label{main-theorem}
Let $H_{0}:\mathbb{R}^{m+1}\rightarrow\mathbb{R}$ be a smooth first integral of system
 \eqref{sys-family-general} with $\e=0$ and  $p_{0}\in \mathcal{NH}(M_{0})$ be a normally
  hyperbolic point. Then there exists a neighborhood $V\subset \mathcal{NH}(M_{0})$ of $p_{0}$ such that
$\varphi_{0}(t) = \big{(} x_{0}(t), \mathbf{y}_{0}(t)\big{)}, \varphi_{0}(0) = p_{0},$
is a solution of the slow system \eqref{diff-algebraic-eq} in $V$ if, and only if, $\varphi_{0}(t)$ 
is a solution of \eqref{sys-first-integral0}
\end{proposition}

In other words, Proposition \ref{main-theorem} says that $\varphi_{0}$ is an orbit on the normally hyperbolic
part of the slow manifold if, and only if, $\varphi_{0}$ is an orbit of \eqref{sys-first-integral0} on 
the level $H_{0}=0$.\\

Now consider the singularly perturbed system
\begin{equation}\label{sys-spp}
\e\dot{x} = H_{\e}(x,\mathbf{y}), \ \dot{\mathbf{y}} = \beta_{\e}(x,\mathbf{y}).
\end{equation}

We remark that $M_{0}$ is the slow manifold of \eqref{sys-spp}.
Take $p_{0} = (x_{0}, \mathbf{y}_{0})\in\mathcal{NH}(M_{0})$. Since $H_{\e}\rightarrow H_{0}$ in 
the $C^{1}$-topology we have
$$
H_{\e}(p_{0})\rightarrow H_{0}(p_{0}), \ \frac{\partial^{i} H_{\e}}{\partial x^{i}}(p_{0}) \rightarrow \frac{\partial^{i} H_{0}}{\partial x^{i}}(p_{0}), \ \frac{\partial^{i} H_{\e}}{\partial \mathbf{y}^{i}}(p_{0}) \rightarrow \frac{\partial^{i} H_{0}}{\partial \mathbf{y}^{i}}(p_{0}).
$$

\noindent\textbf{Remark.} Let $N_{\e}$ be the family of locally invariant surface of 
\eqref{sys-spp} given by Fenichel's Theorem \ref{teo-Fenichel} converging to a compact 
subset $N_{0}\subset\mathcal{NH}\big{(}M_{0}\big{)}$. Then there is a compact subset 
$\widetilde{M}_{\e}\subset\mathcal{NH}\big{(}M_{\e}\big{)}$ diffeomorphic to $N_{\e}$, for $\e$ sufficiently small.\\

From now on we denote
 $\mathcal{NH}\big{(}M_{\e}\big{)}= \big{\{}(x,\mathbf{y})\in M_{\e}: (H_{\e})_x\neq 0\big{\}}.$
The next statements aim to relate the dynamics of \eqref{sys-first-integral} on 
$\mathcal{NH}\big{(}M_{\e}\big{)}$ with the dynamics of the singular perturbation problem 
\eqref{sys-spp}. We will always suppose that $\mathcal{NH}\big{(}M_{0}\big{)}$ and $V$ 
is the neighborhood given by Proposition \ref{main-theorem}. The idea is to use the normal 
hyperbolicity  and Fenichel Theorem \ref{teo-Fenichel} to study the persistence of equilibrium 
points and periodic orbits.

\begin{proposition}\label{prop-spp-equilibrio}
	Let $p_{0}\in N_{0}$ be an equilibrium point of \eqref{sys-first-integral} (for $\e = 0$). Then
\begin{description}
  \item[(a)] For $\e > 0$ sufficiently small, there is a sequence of points 
  $p_{\e}\in \mathcal{NH}\big{(}M_{\e}\big{)}$ which converges to $p_{0}$ and, 
  for each $\e$, $p_{\e}$ is an equilibrium point of \eqref{sys-first-integral}.
  \item[(b)] $p_{\e}\in \mathcal{NH}\big{(}M_{\e}\big{)}$ is an equilibrium point of 
  \eqref{sys-first-integral} (for $\e > 0$) if, and only if, $p_{\e}$ is an equilibrium
   point for \eqref{sys-spp}.
\end{description}
\end{proposition}
\begin{proof}
Since $p_{0}$ is an equilibrium point, it follows from Fenichel's Theorem that for $\e > 0$ 
sufficiently small, there is a sequence of points $p_{\e}\in N_{\e}$ which converges to 
$p_{0}$ such that $p_{\e}$ is an equilibrium point of the singularly perturbed problem 
\eqref{sys-spp}. The manifold $N_{\e}$ is locally invariant and normally hyperbolic for \eqref{sys-spp}.

In particular, for each $\e$ we have $\beta_{\e}(p_{\e}) = \gamma_{\e}(p_{\e}) = 0$ and 
thus $p_{\e}$ is an equilibrium point for \eqref{sys-first-integral}. Since $H_{\e}(p_{\e}) = 0$
 and  $p_{\e}\in N_{\e}$ we have $(H_{\e})_x(p_{\e})\neq 0$ for $\e$ suficiently small. 
 Thus $p_{\e}\in\mathcal{NH}\big{(}M_{\e}\big{)}$ and statement \emph{(a)} is true. 
 Statement \emph{(b)} follows directly.
\end{proof}

For the next result, we will suppose that \eqref{sys-spp} is a singularly perturbed system of the form
\begin{equation}\label{sys-spp-particularcase}
\e\dot{x} = H_{\e}(x,\mathbf{y}), \ \dot{\mathbf{y}} = \beta_{\e}(\mathbf{y}),
\end{equation}
that is, we require that $\beta_{\e}$ does not depend on $x$. This assumption is not a 
simple convenience. In fact, since \eqref{sys-spp} is well defined in the normally hyperbolic 
part of the slow manifold, by the Implicit Function Theorem there is a function $\Phi_{\e}$ 
such that $\Phi_{\e}(\mathbf{y}) = x$.

We also suppose that  \eqref{sys-first-integral} is a system of the form
\begin{equation}\label{sys-first-integral-particularcase}
\dot{x} = -\Bigg{(}\frac{\beta_{\e}(H_{\e})_{\mathbf{y}}}{(H_{\e})_x}\Bigg{)}, \ \dot{y} = \beta_{\e}(\mathbf{y}).
\end{equation}

\begin{proposition}\label{prop-spp-periodico}
Let $\{\varphi_{0}(t)\}\subset V$ be a periodic orbit of \eqref{sys-first-integral-particularcase} (for $\e = 0$). Then
\begin{description}
  \item[(a)] For $\e > 0$ sufficiently small, there is a sequence of orbits 
  $\{\overline{\varphi}_{\e}(t)\}\subset\mathcal{NH}\big{(}M_{\e}\big{)}$ which converges to 
  $\{\varphi_{0}(t)\}$, according Hausdorff distance, such that for each $\e$, 
  $\{\overline{\varphi}_{\e}(t)\}$ is a periodic orbit of \eqref{sys-first-integral-particularcase}.
  \item[(b)] If $\{\psi_{\e}(t)\}\subset \mathcal{NH}(M_{\e})$ is a periodic orbit of
   \eqref{sys-first-integral-particularcase} for $\e > 0$, then $\{\psi_{\e}(t)\}$ is 
   a periodic orbit of \eqref{sys-spp-particularcase}.
\end{description}

\end{proposition}
\begin{proof}
	Let $\varphi_{0}(t) = \Big{(}x_{0}(t), \mathbf{y}_{0}(t)\Big{)}$ be a periodic orbit of 
	\eqref{sys-first-integral-particularcase} (for $\e = 0$) in $V$. Thus $\varphi_{0}(t)$
	 is a periodic orbit for \eqref{sys-spp-particularcase}, for $\e=0$, by Proposition
	  \ref{main-theorem}. Note that $\{\varphi_{0}(t)\}$ is compact, then it follows from Fenichel's
	   Theorem that for $\e$ sufficiently small, there is a sequence of periodic orbits 
	   $\varphi_{\e}(t) = \Big{(}x_{\e}(t), \mathbf{y}_{\e}(t)\Big{)}$ for \eqref{sys-spp-particularcase} 
	   which converges to $\varphi_{0}$ and such that $\{\varphi_{\e}\}\subset N_{\e}$.

On $\mathcal{NH}\big{(}M_{\e}\big{)}$, we have $(H_{\e})_x \neq 0$. By the Implict Function 
Theorem, there is a function $\Phi_{\e}$ such that we can write $\mathcal{NH}\big{(}M_{\e}\big{)}$ 
as the graphic of $\Phi_{\e}$. Define $\overline{\varphi}_{\e}(t) = \Big{(}\Phi_{\e}\circ \mathbf{y}_{\e}(t), \mathbf{y}_{\e}(t)\Big{)}$. 
Since $\varphi_{\e}$ is periodic, $\mathbf{y}_{\e}(t)$ is periodic and then 
$\overline{\varphi}_{\e}(t)$ is periodic. Moreover, $\overline{\varphi}_{\e}\subset \mathcal{NH}\big{(}M_{\e}\big{)}$ 
because $\overline{\varphi}_{\e}$ is contained in the graphic of $\Phi_{\e}$. We also have
 that $\overline{\varphi}_{\e}(t)$ is solution of \eqref{sys-first-integral-particularcase} because 
 $\dot{\mathbf{y}} = \beta_{\e}\Big{(}\overline{\varphi}_{\e}(t)\Big{)}$ and 
 $H_{\e}\circ\overline{\varphi}_{\e}(t) = 0$ implies
$\dot{x} = -\frac{\beta_{\e}(H_{\e})_y + \gamma_{\e}(H_{\e})_z}{(H_{\e})_x}.$

Finally, we have that $\{\overline{\varphi}_{\e}\}$ converges to $\{\varphi_{0}\}$ because 
$\Phi_{\e}$ converges to $\Phi_{0}$ and we can write 
$\varphi_{0}(t) = \Big{(}\Phi_{0}(\mathbf{y}_{0}(t)), \mathbf{y}_{0}(t)\Big{)}$. 
For item \emph{(b)}, observe that $\dot{\mathbf{y}} = \beta_{\e}\Big{(}\psi_{\e}(t)\Big{)}$ and $H_{\e}\circ\psi_{\e}(t) = 0$, 
and then $\{\psi_{\e}(t)\}$ is periodic of \eqref{sys-spp-particularcase}. 
This completes the proof.
\end{proof}

\noindent\textbf{Example.}
Consider
$H(x,y,z,\e) = x - \Big{(}\frac{y^2 + z^2}{2}\Big{)} - \e h(x,y)$,  $\beta(y,z,\e) = -z + y(1 - y^2 - z^2) + \e,$ 
$\gamma(y,z,\e) = y+z(1 - y^2 - z^2) + \e,$
where $h:\mathbb{R}^{2}\rightarrow\mathbb{R}$ is smooth. Then $H_{\e}$ is a first integral of
\begin{equation}\label{exe-familia-int-primeira}
\dot{x} = \beta.(y + \e h_{y}) + \gamma.(z + \e h_{z}), \ \dot{y} = \beta, \ \dot{z} = \gamma.
\end{equation}

For $\e = 0$, by Proposition \ref{main-theorem} the flow on $M_{0}$ is given by the algebraic differential equation
$0 = x - \Big{(}\frac{y^2 + z^2}{2}\Big{)}$, $\dot{y} = -z + y(1 - y^2 - z^2),$  $\dot{z} = y+z(1 - y^2 - z^2).$

Note that there is a stable limit cycle on $M_{0}$ and the origin is an equilibrium point.
 It follows from Propositions \ref{prop-spp-equilibrio} and \ref{prop-spp-periodico} that for $\e > 0$ 
 sufficiently small such compact orbits persist for system \eqref{exe-familia-int-primeira}. 
 This fact also can be checked noticing that the system
$\dot{y} = -z + y(1 - y^2 - z^2)$, $\dot{z} = y+z(1 - y^2 - z^2)$
is structurally stable.

	\section{Acknowledgments}
	Paulo R. da Silva is partially supported by CAPES  and FAPESP. Ot\'{a}vio H. Perez is partially supported by FAPESP. 
	.

\end{document}